\documentclass[12pt]{amsart}

\usepackage{amsfonts,graphics,amsmath,amsthm,amsfonts,amscd,amssymb,amsmath,latexsym,multicol,mathrsfs}
\usepackage{epsfig,url}
\usepackage{flafter}
\usepackage{fancyhdr}
\usepackage{color}
\usepackage{hyperref}
\usepackage[all]{xy}
\hypersetup{colorlinks=true, linkcolor=black}
\numberwithin{equation}{section}

\newtheorem{corollary}[subsection]{Corollary}
\newtheorem{lemma}[subsection]{Lemma}
\newtheorem{proposition}[subsection]{Proposition}
\newtheorem{theorem}[subsection]{Theorem}

\newtheorem{question}[subsection]{Question}

\theoremstyle{definition}
\newtheorem{definition}[subsection]{Definition}

\newtheorem{example}[subsection]{Example}
\newtheorem{remark}[subsection]{Remark}

\newcommand{\PP}{\mathbb P}
\newcommand{\KK}{\mathbb K}
\newcommand{\ZZ}{\mathbb Z}
\newcommand{\QQ}{\mathbb Q}
\newcommand{\LLL}{\mathscr{L}}
\newcommand{\HHH}{\mathscr{H}}
\newcommand{\GGG}{\mathscr{G}}

\DeclareMathOperator{\Aut}{Aut}
\DeclareMathOperator{\AAut}{\mathbf{Aut}}
\DeclareMathOperator{\Bir}{Bir}
\DeclareMathOperator{\Ine}{Ine}

\DeclareMathOperator{\Spec}{Spec}
\DeclareMathOperator{\Hom}{Hom}
\DeclareMathOperator{\Gal}{Gal}
\DeclareMathOperator{\Pic}{Pic}

\DeclareMathOperator{\red}{red}
\DeclareMathOperator{\GL}{GL}
\DeclareMathOperator{\PGL}{PGL}
\DeclareMathOperator{\PSL}{PSL}
\DeclareMathOperator{\SL}{SL}
\DeclareMathOperator{\NS}{NS}

\DeclareMathOperator{\Hilb}{Hilb}
\DeclareMathOperator{\Tot}{Tot}

\DeclareMathOperator{\one}{\mathbf{1}}

\def \ge {\geqslant}
\def \le {\leqslant}

\title{Automorphisms of surfaces over fields of positive characteristic}
\author{Yifei Chen and Constantin Shramov}

\address{\emph{Yifei Chen}
\newline
\textnormal{
Hua Loo-Keng Key Laboratory of Mathematics, Academy of Mathematics and Systems Science,
Chinese Academy of Sciences, No. 55 Zhonguancun East Road, Haidian District, Beijing, 100190, P.R.China.}
\newline
\textnormal{\texttt{yifeichen@amss.ac.cn}}
}

\address{\emph{Constantin Shramov}
\newline
\textnormal{Steklov Mathematical Institute of RAS,
8 Gubkina street, Moscow 119991, Russia.
}
\newline
\textnormal{%National Research University Higher School of Economics,
National Research University Higher School of Economics, Laboratory of Algebraic Geometry, NRU HSE, 6 Usacheva str., Moscow, 117312, Russia.
}
\newline
\textnormal{\texttt{costya.shramov@gmail.com}}}

\begin{document}
\maketitle

\begin{abstract}
We study automorphism and birational automorphism groups
of varieties over fields of positive
characteristic from the point of view of Jordan and $p$-Jordan property.
In particular, we show that the Cremona group of rank $2$
over a field of characteristic $p>0$ is $p$-Jordan,
and the birational automorphism group of an arbitrary geometrically irreducible algebraic surface is nilpotently $p$-Jordan of class at most~$2$.
Also, we show that the automorphism group of
a smooth geometrically irreducible projective variety of non-negative Kodaira
dimension is Jordan in the usual sense.
\end{abstract}

\tableofcontents

\section{Introduction}

Birational automorphism groups of algebraic varieties sometimes have
extremely complicated structure. Also, their finite subgroups
may be hard to classify explicitly. To obtain an approach to
the description of finite subgroups, the following definition
was introduced by V.\,Popov.

\begin{definition}[{see \cite[Definition~2.1]{Popov}}]
\label{definition:Jordan}
A group~$\Gamma$ is called \emph{Jordan}
(alternatively, one says that~$\Gamma$ has \emph{Jordan property}),
if there exists a constant~\mbox{$J=J(\Gamma)$},
depending only on~$\Gamma$,
such that any finite subgroup of~$\Gamma$
contains a normal abelian subgroup of index at most~$J$.
\end{definition}

A classical theorem of C.\,Jordan
asserts that the general linear group, and thus also every linear algebraic group, over a field of characteristic zero
is Jordan;
see~\mbox{\cite[\S40]{Jordan}}, or~\cite{Bieberbach}, or~\cite{Frobenius},
or~\mbox{\cite[Theorem~36.13]{Curtis-Reiner-1962}}, or~\cite[Theorem~A]{Robinson},
or~\cite[\S4]{Tao}, or~\cite{Breuillard}, or~\cite[\S2]{BreulliardGreen}, or~\mbox{\cite[\S3]{MiR-Jordan}}, or~\mbox{\cite[Theorem~9.9]{Serre-FiniteGroups}}.
J.-P.\,Serre proved in~\mbox{\cite[Theorem~5.3]{Serre-2009}}
that the same property holds for the Cremona group of rank~$2$,
that is, for the group~\mbox{$\Bir(\PP^2)$}
of birational automorphisms of the projective plane,
over fields of characteristic zero. V.\,Popov
classified in~\mbox{\cite[Theorem~2.32]{Popov}} all
surfaces over fields of characteristic zero
whose birational automorphism
group is Jordan. Sh.\,Meng and D.-Q.\,Zhang proved
in~\mbox{\cite[Theorem~1.6]{MZ}}
that Jordan property holds for automorphism groups of projective varieties over fields of characteristic zero.

However, all this fails miserably over fields of positive characteristic.
Indeed, let $p$ be a prime number, let $\mathbf{F}_{p^k}$ denote the field of $p^k$ elements, and
let~$\bar{\mathbf{F}}_p$ be the algebraic closure of the field~$\mathbf{F}_p$. Then the group $\PGL_2(\bar{\mathbf{F}}_p)$ contains the
groups~\mbox{$\mathrm{PSL}_2(\mathbf{F}_{p^k})$}
for all positive integers $k$; the latter finite groups are simple apart from a (small) finite number of exceptions, see \cite[\S3.3.1]{Wilson}.
This means that the group~\mbox{$\PGL_2(\bar{\mathbf{F}}_p)$}
not only fails to be Jordan, but moreover, its finite subgroups cannot contain \emph{any} proper subgroups
of bounded index. In particular, no straightforward generalizations of
Jordan property, like nilpotent Jordan property (see~\cite{Guld}) or solvable
Jordan property (see \cite[\S8]{ProkhorovShramov-Bir}), can hold in this case.

The above observation naturally leads to the following definitions.

\begin{definition}[{\cite[Definition~1.2]{Hu}}]
\label{D:g-Jordan-1}
Let $p$ be a prime number. A group~$\Gamma$ is called \emph{generalized} $p$-\emph{Jordan}, if there is a constant $J(\Gamma)$, depending only on $\Gamma$, such that every
finite subgroup $G$ of $\Gamma$ whose order is not divisible by $p$ contains a normal abelian subgroup
of index at most $J(\Gamma)$.
\end{definition}

\begin{definition}[{\cite[Definition~1.6]{Hu}, see also \cite{BrauerFeit}}]
\label{D:g-Jordan-2}
Let $p$ be a prime number. A group $\Gamma$ is called
$p$-\emph{Jordan}, if there exist constants
$J(\Gamma)$ and~\mbox{$e(\Gamma)$}, depending only on $\Gamma$, such that every finite subgroup
$G$ of $\Gamma$ contains a normal abelian subgroup of order coprime
to $p$ and index at most~\mbox{$J(\Gamma)\cdot |G_{p}|^{e(\Gamma)}$},
where~$G_{p}$ is a $p$-Sylow subgroup of $G$.
\end{definition}

Note that every Jordan group is $p$-Jordan
(see Corollary~\ref{corollary:Jordan-vs-p-Jordan}),
and every $p$-Jordan group is generalized $p$-Jordan,
but the converse statements do not hold.
It was proved in~\mbox{\cite[Theorem~5.3]{Serre-2009}} that the Cremona group
of rank~$2$ over a field of characteristic $p>0$ is generalized $p$-Jordan.
As for the $p$-Jordan property,
the following analog of the theorem of C.\,Jordan
is known.

\begin{theorem}[{\cite{BrauerFeit}, \cite[Theorem~0.4]{LarsenPink}}]
\label{theorem:BrauerFeit}
Let $n$ be a positive integer. Then there
exists a constant $J(n)$ such that for every prime $p$ and every
field $\Bbbk$ of characteristic $p$, every finite subgroup
$G$ of~\mbox{$\GL_n(\Bbbk)$}
contains a normal abelian subgroup of order coprime
to $p$ and index at most~\mbox{$J(n)\cdot |G_{p}|^{3}$},
where $G_{p}$ is a $p$-Sylow subgroup of $G$.
In particular, for every field $\Bbbk$ of characteristic $p>0$
the group~\mbox{$\GL_n(\Bbbk)$} is $p$-Jordan.
\end{theorem}

It immediately follows from Theorem~\ref{theorem:BrauerFeit}
that (the group of $\Bbbk$-points of) every linear algebraic group over a field
$\Bbbk$ of characteristic $p>0$
is $p$-Jordan.

It appears that the $p$-Jordan property is useful to study automorphism groups
and birational automorphism groups of algebraic varieties (by which we mean geometrically
reduced separated schemes of finite type)
over fields of positive characteristic.
The following analog of \cite[Theorem~1.6]{MZ} was proved by F.\,Hu.

\begin{theorem}[{\cite[Theorem~1.10]{Hu}}]
\label{theorem:Hu}
Let $\Bbbk$ be a field of characteristic $p>0$, and
let $X$ be a projective variety over $\Bbbk$.
Then the automorphism group $\Aut(X)$ is $p$-Jordan.
\end{theorem}

\medskip
The purpose of this paper
is to initiate a systematic study of $p$-Jordan and generalized $p$-Jordan
properties for groups of birational automorphisms of varieties
over fields of positive characteristic,
and to generalize to this setting
the relevant results for surfaces over fields of characteristic zero
(as usual, by a surface we mean a variety of dimension~$2$).
Our first goal is to prove an
analog of~\mbox{\cite[Theorem~5.3]{Serre-2009}}.

\begin{theorem}\label{theorem:Serre}
There exists a constant $J$
such that for every prime $p$ and every field $\Bbbk$ of characteristic $p$,
every finite subgroup
$G$ of the birational automorphism group~\mbox{$\Bir(\PP^2)$},
contains a normal abelian subgroup of order coprime
to $p$ and index at most~\mbox{$J\cdot |G_{p}|^{3}$},
where~$G_{p}$ is a $p$-Sylow subgroup of $G$.
In particular, for every field $\Bbbk$ of characteristic~\mbox{$p>0$}
the group $\Bir(\PP^2)$ is $p$-Jordan.
\end{theorem}

Furthermore, we obtain an analog of \cite[Theorem~2.32]{Popov}
(see also~\mbox{\cite[Theorem~1.7]{ProkhorovShramov-CCS}}).

\begin{theorem}\label{theorem:Popov}
Let $\Bbbk$ be an algebraically closed field of characteristic~\mbox{$p>0$}, and let~$S$ be an irreducible algebraic
surface over $\Bbbk$. The following assertions hold.

\begin{itemize}
\item[(i)] If $S$ is birational to a product $E\times\PP^1$ for some elliptic
curve $E$, then the group
$\Bir(S)$ is not generalized $p$-Jordan.

\item[(ii)] If the Kodaira dimension of $S$ is negative but $S$ is not birational
to a product $E\times\PP^1$ for any elliptic curve $E$, then the group
$\Bir(S)$ is $p$-Jordan but not Jordan.

\item[(iii)] If the Kodaira dimension of $S$ is non-negative, then the group
$\Bir(S)$ is Jordan.
\end{itemize}
\end{theorem}

In the proof of Theorem~\ref{theorem:Popov} we use the following
assertion which makes Theorem~\ref{theorem:Hu} more precise in one important particular case.

\begin{proposition}\label{proposition:Hu-improved}
Let $\Bbbk$ be an arbitrary field, and
let $X$ be a smooth geometrically irreducible
projective variety of non-negative Kodaira dimension over~$\Bbbk$.
Then the group $\Aut(X)$ is Jordan.
\end{proposition}

Our next result is an analog of \cite[Proposition~1.6]{ProkhorovShramov-2019}.
It can be regarded as (a little bit more precise version of) a particular
subcase of Theorem~\ref{theorem:Popov}(iii).

\begin{proposition}
\label{proposition:Jordan-kappa-0}
There exists a constant $J$ such that
for every field $\Bbbk$ and
every geometrically irreducible algebraic
surface $S$  of Kodaira dimension~$0$ over~$\Bbbk$,
every finite subgroup of~\mbox{$\Bir(S)$}
contains a normal abelian subgroup of
index at most~$J$.
\end{proposition}

Using the terminology of \cite[Definition~1.6]{ProkhorovShramov-Bir},
one can reformulate Proposition~\ref{proposition:Jordan-kappa-0} by saying that
the set of automorphism groups of all geometrically irreducible algebraic
surfaces of Kodaira dimension~$0$ over all fields is uniformly Jordan.

Given a variety $X$ and a point $P\in X$, we denote by $\Aut(X;P)$ the stabilizer of $P$
in the group $\Aut(X)$.
The following result is an analog
of~\mbox{\cite[Proposition~1.3]{ProkhorovShramov-2019}}.

\begin{proposition}
\label{proposition:stabilizer-kappa-0}
There exists a constant $B$ such that for every field $\Bbbk$, every smooth geometrically irreducible projective surface
$S$ of Kodaira dimension~$0$ over $\Bbbk$, every
$\Bbbk$-point $P\in S$, and every finite subgroup~\mbox{$G\subset \Aut(S;P)$}
the order of the group~$G$ is at most~$B$.
\end{proposition}

We point out that the assertion of Proposition~\ref{proposition:stabilizer-kappa-0}
fails in general for stabilizers of closed points on varieties over algebraically non-closed fields
which are not $\Bbbk$-points,
see Example~\ref{example:large-stabilizer}.
However, it holds if one replaces the stabilizers by the inertia groups,
see Corollary~\ref{corollary:inertia}. Also, there is the following (partial) generalization of Proposition~\ref{proposition:stabilizer-kappa-0}
which holds for varieties of arbitrary dimension over arbitrary fields.
It is an analog of~\mbox{\cite[Theorem~1.5]{ProkhorovShramov-2019}}.

\begin{theorem}\label{theorem:high-dimension-stabilizer}
Let $\Bbbk$ be a field, and let $X$
be a smooth geometrically irreducible projective variety
of non-negative Kodaira dimension over~$\Bbbk$.
Then there exists a constant~\mbox{$B=B(X)$}
such that for every closed point $P\in X$ and
every finite subgroup~\mbox{$G\subset\Aut(X;P)$}
the order of the group $G$ is at most~$B$.
\end{theorem}

Since some of the birational automorphism groups are not Jordan, one can weaken
it by considering nilpotent subgroups instead of abelian ones in
Definition~\ref{definition:Jordan}.
Recall that a group $G$ is said to be \emph{nilpotent of class
at most~$c$} if its upper central series has length at most $c$,
see \cite[\S1D]{Isaacs} for details; in particular, nilpotent groups of class
at most $1$ are exactly abelian groups, and the only nilpotent group
of class at most~$0$ is the trivial group.
This leads to the notion of nilpotently Jordan groups.

\begin{definition}[{see \cite[Definition~1]{Guld}}]
\label{definition:nilp-Jordan}
A group~$\Gamma$ is \emph{nilpotently Jordan of class at most $c$}
if there exists a constant~\mbox{$J(\Gamma)$},
depending only on~$\Gamma$,
such that any finite subgroup of~$\Gamma$
contains a normal subgroup $N$
of index at most~\mbox{$J(\Gamma)$}, where~$N$
is nilpotent of class at most~$c$.
\end{definition}

We introduce the analogs of this definition
suitable for automorphism groups of varieties over fields of positive
characteristic.

\begin{definition}
\label{D:g-nilp-Jordan-1}
Let $p$ be a prime number.
A group $\Gamma$ is \emph{generalized nilpotently $p$-Jordan of
class at most
$c$},
if there is a constant $J(\Gamma)$, depending only on $\Gamma$,
such that every
finite subgroup of $\Gamma$
whose order is not divisible by~$p$ contains a normal subgroup $N$
of index at most $J(\Gamma)$, where~$N$ is nilpotent of class at most $c$.
\end{definition}

\begin{definition}
\label{D:g-nilp-Jordan-2}
Let $p$ be a prime number. A group $\Gamma$ is \emph{nilpotently
$p$-Jordan of class at most $c$}, if there exist constants $J(\Gamma)$ and $e(\Gamma)$, depending only on $\Gamma$, such that every finite subgroup
$G$ of $\Gamma$ contains a normal subgroup $N$ of
order coprime to $p$ and index at most
$J(\Gamma)\cdot |G_{p}|^{e(\Gamma)}$, where
$N$ is nilpotent of class at most $c$,
and $G_{p}$
denotes a $p$-Sylow subgroup of $G$.
\end{definition}

Similarly to the situation with the usual Jordan property,
every nilpotently Jordan group of class at most $c$ is nilpotently
$p$-Jordan of class at most $c$
(see Corollary~\ref{corollary:Jordan-vs-p-Jordan}),
and every nilpotently $p$-Jordan group
of class at most $c$
is generalized nilpotently $p$-Jordan of class at most $c$, while
the converse statements do not hold.
Note also that a group is generalized nilpotently $p$-Jordan of class at most
$1$ (respectively, nilpotently
$p$-Jordan group of class at most $1$) if and only if
it is generalized $p$-Jordan (respectively, $p$-Jordan).

According to \cite[Theorem~2]{Guld},
the birational automorphism group of any (geometrically irreducible)
variety over
a field of characteristic zero is nilpotently Jordan of class at most~$2$.
We prove an analog of this assertion for surfaces over fields of
positive characteristic.

\begin{theorem}\label{theorem:nilpotent}
Let $\Bbbk$ be a field of characteristic $p>0$.
Let $S$ be a geometrically irreducible algebraic
surface over $\Bbbk$.
Then the group $\Bir(S)$ is nilpotently
$p$-Jordan of class at most~$2$.
\end{theorem}

\medskip
The plan of our paper is as follows.
In Section~\ref{section:preliminaries}
we collect some elementary assertions about groups and lattices used in the rest of the paper.
In Section~\ref{section:Aut}
we recall the basic concepts and facts concerning
automorphism groups and schemes of projective varieties.
In Section~\ref{section:MZ} we discuss the group of connected components
of the automorphism group scheme of a projective variety following~\cite{MZ}.
In Section~\ref{section:surfaces}
we recall the basics of the Minimal Model Program in dimension~$2$,
including some theorems from its equivariant version.
In Section~\ref{section:abelian-varieties}
we collect auxiliary facts about automorphism groups of abelian varieties.
In Section~\ref{section:varieties-kappa-non-negative}
we make some observations on
automorphism groups of varieties of non-negative
Kodaira dimension and
prove Proposition~\ref{proposition:Hu-improved}.
In Section~\ref{section:curves}
we discuss automorphism groups of smooth projective curves.
In Section~\ref{section:Cremona}
we prove Theorem~\ref{theorem:Serre}.
In Section~\ref{section:Popov}
we prove Theorem~\ref{theorem:Popov}.
In Section~\ref{section:kappa-0}
we study automorphism groups of surfaces of zero Kodaira dimension and
prove Propositions~\ref{proposition:Jordan-kappa-0} and~\ref{proposition:stabilizer-kappa-0}.
In Section~\ref{section:high-dim}
we prove Theorem~\ref{theorem:high-dimension-stabilizer}.
In Section~\ref{section:nilpotent}
we prove Theorem~\ref{theorem:nilpotent}.
In Section~\ref{section:discussion}
we discuss some open questions concerning
(birational) automorphism groups of varieties over fields of positive
characteristic.

In some cases, the proofs of our main results go along the same lines as the proofs of the corresponding results
in characteristic~$0$. This applies to Propositions~\ref{proposition:Jordan-kappa-0} and~\ref{proposition:stabilizer-kappa-0},
Theorem~\ref{theorem:high-dimension-stabilizer} and, to a certain extent, to
Theorem~\ref{theorem:Popov}.
The proof of Theorem~\ref{theorem:Serre}
mostly follows the proof of~\mbox{\cite[Theorem~5.3]{Serre-2009}} but contains additional arguments
needed to treat finite subgroups of~\mbox{$\Bir(\PP^2)$}
whose orders are divisible by the characteristic of the field.
Proposition~\ref{proposition:Hu-improved} and its proof look new in the context of positive characteristic.
Theorem~\ref{theorem:nilpotent} (as well as the accompanying Definitions~\ref{D:g-nilp-Jordan-1} and~\ref{D:g-nilp-Jordan-2})
and its proof are also entirely new; we point out that the proof does not follow the ideas of~\cite{Guld}, but is rather inspired by the approach of J.-P.\,Serre to~\mbox{\cite[Theorem~5.3]{Serre-2009}}.
Many statements collected in the preliminary sections of our paper are well known to experts and are widely
used at least over fields of characteristic zero, but are not readily available in the literature in the positive characteristic setup.
Since one of the goals of our paper is to provide a survey of the methods of studying finite
groups of birational automorphisms in arbitrary characteristic, in
such cases we take the opportunity to spell out the details of the proofs.
This allows either to emphasize
that the proofs do not depend on the characteristic of the base field,
see e.g. Lemmas~\ref{lemma:regularization} and~\ref{lemma:MZ} and
Corollary~\ref{corollary:general-type};
or to be able to mention the (minor) differences in the proofs appearing in the case of positive characteristic, see e.g.
Theorem~\ref{theorem:fixed-point}.

\medskip
Throughout the paper, we use the following standard notation.
Given a field~$\Bbbk$, we denote by $\bar{\Bbbk}$ its algebraic closure.
If $\Bbbk\subset\KK$ is a field extension, and~$X$ is a scheme defined over~$\Bbbk$,
we denote by $X_{\KK}$ the extension of scalars of~$X$ to~$\KK$.
By a $\Bbbk$-point of a scheme $X$ over $\Bbbk$ we mean its closed point of degree~$1$;
the set of all $\Bbbk$-points of $X$ is denoted by~$X(\Bbbk)$.

\medskip
We are grateful to Yi Gu, Yuri Prokhorov, Dmitry Timashev,
Vadim Vologodsky, and Jinsong Xu for useful discussions.
Yifei Chen was partially supported by
NSFC grants (No.~11688101, 11771426).
The work of Constantin Shramov was performed at the Steklov
International Mathematical Center and supported by the Ministry of
Science and Higher Education of the Russian Federation (agreement no.~\mbox{075-15-2022-265}).
Constantin Shramov was also supported by
the HSE University Basic Research Program,
Russian Academic Excellence Project~\mbox{``5-100''},
and by the Foundation for the
Advancement of Theoretical Physics and Mathematics ``BASIS''.

\section{Preliminaries}
\label{section:preliminaries}

In this section we collect auxiliary facts about groups and lattices.

\medskip
\textbf{Group theory.}
Recall that a subgroup of a group $G$ is called \emph{characteristic}
if it is preserved by all automorphisms of~$G$.

\begin{example}\label{example:normal-Sylow-subgroup}
Let $G$ be a finite group that has a normal $p$-Sylow subgroup~$G_p$.
Then $G_p$ is characteristic in~$G$. Indeed, it consists of all
elements of $G$ whose order is a power of~$p$.
\end{example}

\begin{example}\label{example:characteristic-product}
Let $G\cong G'\rtimes G''$ be a finite group. Suppose that
the orders of~$G'$ and $G''$ are coprime. Then $G'$
is a characteristic subgroup of $G$. Indeed, it consists of all
elements of $G$ whose order divides~$|G'|$.
\end{example}

\begin{example}\label{example:characteristic-by-generators}
Let $H$ be an abelian group generated by its elements~\mbox{$h_1,\ldots,h_k$},
and let $r$ be a positive integer. Then the subgroup
$H'$ generated by~\mbox{$h_1^r,\ldots,h_k^r$} coincides with the subgroups that consists of
$r$-th powers of elements of $H$. Therefore, $H'$ is a characteristic subgroup
of index at most~$r^k$ in~$H$.
\end{example}

\begin{remark}\label{remark:normal-vs-characteristic}
Let $G$ be a group, let $H$ be its normal subgroup, and let $F$ be a normal
subgroup of $H$. Then $F$ is not necessarily normal in~$G$.
However, if $F$ is characteristic in $H$, then it is normal in~$G$.
\end{remark}

It appears that the Jordan property implies the
$p$-Jordan property for every~$p$. To see this, we make an auxiliary observation.

\begin{lemma}\label{lemma:Jordan-vs-p-Jordan}
Let $p$ be a prime number, let $J$ be a positive integer, and let $G$ be a finite group.
Let $\tilde{A}$ be a normal subgroup of index at most $J$ in $G$. Suppose that $\tilde{A}$ is
abelian (respectively, nilpotently Jordan of class at most $c$).
Then~$G$ contains a normal subgroup $A$
such that the minimal number of elements that generate $A$ does not exceed that of
$\tilde{A}$, the order of $A$ is coprime to $p$, the index of $A$ in $G$ is at most
$J\cdot |G_p|$, where $G_p$ is a $p$-Sylow subgroup of $G$,
and~$A$ is abelian (respectively, nilpotently Jordan of class at most $c$).
\end{lemma}
\begin{proof}
Let $\tilde{A}_p$ be the $p$-Sylow subgroup of $\tilde{A}$.
Then
$$
\tilde{A}\cong \tilde{A}_p\times A,
$$
where the group $A$
is isomorphic to the product of the $q$-Sylow subgroups
of~$\tilde{A}$ for all prime numbers $q$ different from~$p$,
see \cite[Theorem~1.26]{Isaacs}.
In particular, the group~$A$ is abelian (respectively,
nilpotent of class at most~$c$), and its order is coprime to~$p$.
Also, if $\tilde{A}$ can be generated by the elements $\tilde{a}_1,\dots, \tilde{a}_r$,
then the projections of these elements to $A$ generate $A$.
Note that~$A$ is a characteristic subgroup of~$\tilde{A}$
by Example~\ref{example:characteristic-product}.
Thus, $A$ is normal in~$G$. Finally, one can see that the index
of~$A$ in $G$ is at most
$$
J\cdot |\tilde{A}_p|\le J\cdot |G_p|.
$$
\end{proof}

\begin{corollary}\label{corollary:Jordan-vs-p-Jordan}
Let $p$ be a prime number, and let $\Gamma$ be a group. Suppose that~$\Gamma$ is
Jordan (respectively, nilpotently Jordan of class at most $c$).
Then $\Gamma$ is $p$-Jordan (respectively, nilpotently $p$-Jordan of class at most $c$).
\end{corollary}
\begin{proof}
Let $G$ be a finite subgroup of $\Gamma$. Then $G$ contains a normal
subgroup~$\tilde{A}$ such that its index is
bounded by some constant $J=J(\Gamma)$
independent of~$G$, and $\tilde{A}$ is abelian (respectively,
nilpotent of class at most~$c$). Therefore, the assertion follows from
Lemma~\ref{lemma:Jordan-vs-p-Jordan}.
\end{proof}

The following fact is standard.

\begin{lemma}\label{lemma:conjugation}
Let $G$ be a finite group, and let $G'$ be its normal subgroup.
Let~\mbox{$A'\subset G'$} be a subgroup that is normal in $G'$.
Denote by $B$ the index of~$G'$ in~$G$, and by~$J$ the index of
$A'$ in $G'$. Then $A'$ contains a subgroup $A$
such that~$A$ is normal in~$G$, and the index of $A$ in $G$ is at most~$BJ^B$.
\end{lemma}
\begin{proof}
The group~$G$ acts on $G'$ by conjugation,
and the conjugation by the elements of~$G'$ preserve~$A'$.
Let $A'_1=A',\ldots,A'_r$ be the orbit of $A'$ under this action.
Then~\mbox{$r\le |G/G'|\le B$}, so that the index of the intersection
$$
A=\bigcap\limits_{i=1}^r A'_i
$$
in $G'$ is at most $J^r\le J^B$. Thus, $A$ is a normal
subgroup of index at most~\mbox{$BJ^B$} in $G$.
\end{proof}

Lemma~\ref{lemma:conjugation} can be applied to find normal
subgroups of a given group with the properties that are inherited
by subgroups, like the properties of being abelian or nilpotent.
Moreover, if under the assumptions of Lemma~\ref{lemma:conjugation}
the group $A'$ is abelian, then one can find a normal abelian
subgroup~$A$ in $G$ such that the index of $A$ in $G$ is at most~$BJ^2$,
see~\mbox{\cite[Theorem~1.41]{Isaacs}}. However, as is pointed
out in~\cite[Remark~3.1]{Hu}, in the latter
case we cannot guarantee that $A$ is contained in $A'$, and thus
have no control on divisibility properties of the order of $A$
(which is essential for the notions of a $p$-Jordan group
or a nilpotently $p$-Jordan group).

One says that a group $\Gamma$ has  \emph{bounded finite subgroups}
if there exists a constant~\mbox{$B=B(\Gamma)$} such that every finite subgroup of $\Gamma$ has order at most $B$.
The next lemma allows one to check Jordan-type properties for
certain extensions of groups.

\begin{lemma}[{cf. \cite[Lemma~2.11]{Popov}}]
\label{lemma:group-theory}
Let $p$ be a prime number, and let
$$
1\to \Gamma'\to\Gamma\to\Gamma''
$$
be an exact sequence of groups.
Suppose that $\Gamma''$ has bounded finite subgroups.
Suppose also that the group $\Gamma'$ is  Jordan
(respectively, $p$-Jordan, generalized $p$-Jordan,
nilpotently Jordan of class at most $c$,
nilpotently $p$-Jordan of class at most $c$,
generalized nilpotently $p$-Jordan of class at most $c$).
Then the group~$\Gamma$ is  Jordan (respectively, $p$-Jordan,
generalized $p$-Jordan,
nilpotently Jordan of class at most $c$,
nilpotently $p$-Jordan of class at most $c$,
generalized nilpotently $p$-Jordan of class at most $c$).
\end{lemma}

\begin{proof}
By assumption, we know that there exists a constant $B$ such that every finite
subgroup of $\Gamma''$ has order at most $B$.
Let $G$ be a finite subgroup of $\Gamma$, and let $G'=G\cap \Gamma'$.
Then the index of $G'$ in $G$ is at most $B$.

Suppose that $\Gamma'$ is Jordan, or that
$\Gamma'$ is generalized $p$-Jordan
and the order of~$G$ is coprime to $p$. Then $G'$ contains a normal abelian
subgroup
of index at most $J$ for some constant $J=J(\Gamma')$ that does not
depend on $G'$. Therefore, $G$ contains
a normal abelian subgroup of index at most $BJ^B$ by Lemma~\ref{lemma:conjugation}.

Similarly, suppose that $\Gamma'$ is nilpotently Jordan of class at most $c$,
or that~$\Gamma'$ is generalized nilpotently $p$-Jordan of class at most $c$, and
the order of~$G$ is coprime to $p$. Then $G'$ contains a normal
nilpotent subgroup of class at most $c$ that has
index at most $J$ for some constant $J=J(\Gamma')$. Therefore, $G$ contains
a normal nilpotent subgroup of class at most $c$ that has
index at most $BJ^B$ by Lemma~\ref{lemma:conjugation}.

Now suppose that $\Gamma'$ is $p$-Jordan (respectively,
nilpotently $p$-Jordan of class at most $c$).
Let $G_p$ and $G'_p$ be $p$-Sylow subgroups of~$G$ and $G'$.
The group~$G'$ contains a normal
subgroup~$A'$ of order coprime to $p$ and
index at most~\mbox{$J\cdot |G'_p|^e$} for some constants~\mbox{$J=J(\Gamma')$}
and $e=e(\Gamma')$ such that $A'$ is abelian
(respectively, nilpotent of class at most $c$).
Applying Lemma~\ref{lemma:conjugation}, we find a normal
subgroup~$A$ in $G$ such that~$A$ is abelian (respectively,
nilpotent of class at most $c$), its order is coprime to $p$, and its index in $G$
is at most
$$
B\cdot \big(J\cdot |G'_p|^e\big)^B=BJ^B\cdot |G'_p|^{Be}\le
BJ^B\cdot |G_p|^{Be}.
$$
\end{proof}

The next results are partial analogs of \cite[Lemma~2.8]{ProkhorovShramov-Bir} for $p$-Jordan groups.

\begin{lemma}\label{lemma:Heisenberg-estimate-1}
Let $p$ be a prime number, and let
$$
1\to F\to H\to \bar{H}\to 1
$$
be an exact sequence of finite groups such that $\bar{H}$ is abelian and generated by~$r$ elements.
Suppose that $|F|\le B\cdot |F_p|^e$ for some positive constants~$B$ and~$e$, where $F_p$ is a $p$-Sylow subgroup of $F$, and
suppose that
$F$ is generated by~$s$ elements.
Then $H$ contains a characteristic abelian subgroup of order coprime to $p$ and
index at most~\mbox{$B^{r+s}\cdot |H_p|^{e(r+s)+1}$}, where $H_p$ is a $p$-Sylow subgroup of~$H$.
\end{lemma}

\begin{proof}
First, let us bound the index of the center $Z$ of the group $H$.
Let $K\subset H$ be the commutator subgroup. Since $\bar{H}$ is abelian, we see that
$K$ is contained in $F$.
For every element $x\in H$, denote by $Z(x)$ the centralizer of $x$ in $H$, and by $K(x)$ the set of elements of the form
$hxh^{-1}x^{-1}$ for various $h\in H$. Then the index of $Z(x)$ does not exceed $|K(x)|\le |K|$.

By assumption, one can choose $r+s$ elements $x_1,\ldots,x_{r+s}$ generating $H$.
Thus $Z=Z(x_1)\cap\ldots\cap Z(x_{r+s})$. Hence
\begin{multline*}
[H:Z]\le [H:Z(x_1)]\cdot\ldots\cdot [H:Z(x_{r+s})]\le |K(x_1)|\cdot\ldots\cdot |K(x_{r+s})|\\ \le |K|^{r+s} \le |F|^{r+s}\le B^{r+s}\cdot |F_p|^{e(r+s)}.
\end{multline*}

Now let $Z'$ be the maximal subgroup in $Z$ whose order is coprime to~$p$. Then~$Z'$
is a characteristic abelian subgroup of $H$,
and the index of $Z'$ in $Z$ equals the order of the $p$-Sylow subgroup $Z_p$ of $Z$.
Therefore, we have
\begin{multline*}
[H:Z']=[H:Z]\cdot [Z:Z']\le \left(B^{r+s}\cdot |F_p|^{e(r+s)}\right)\cdot |Z_p|\\ \le B^{r+s}\cdot |F_p|^{e(r+s)}\cdot |H_p|\le B^{r+s}\cdot |H_p|^{e(r+s)+1}.
\end{multline*}
\end{proof}

\begin{corollary}
Let $p$ be a prime number, and let
$$
1\to \Gamma'\to\Gamma\to\Gamma''
$$
be an exact sequence of groups.
Suppose that $\Gamma''$ is $p$-Jordan,
and there exist positive constants $r$, $B$, $e$, and $s$
such that
\begin{itemize}
\item every finite abelian subgroup of $\Gamma''$ whose order is coprime to $p$ is generated by at most~$r$ elements;

\item for every finite subgroup $F$ of $\Gamma'$ one has $|F|\le B\cdot |F_p|^e$, where $F_p$ is a $p$-Sylow subgroup of $F$, and
$F$ is generated by at most $s$ elements.
\end{itemize}
Then the group $\Gamma$ is $p$-Jordan.
\end{corollary}

\begin{proof}
Let $G$ be a finite subgroup of $\Gamma$. Then $G$ fits into an exact sequence
$$
1\to F\to G\to \bar{G}\to 1,
$$
where $F\subset \Gamma'$ and $\bar{G}\subset\Gamma''$. By assumption, there exist positive
constants~$\bar{B}$ and~$\bar{e}$ that do not depend on $\bar{G}$ such that
$\bar{G}$ contains a normal abelian subgroup~$\bar{H}$ whose order is coprime to $p$
and whose index is bounded by $\bar{B}\cdot |\bar{G}_p|^{\bar{e}}$, where~$\bar{G}_p$ is a $p$-Sylow
subgroup of $\bar{G}$. Moreover, $\bar{H}$ can be generated by $r$ elements.
Let~$H$ be the preimage of $\bar{H}$ in $G$. According to Lemma~\ref{lemma:Heisenberg-estimate-1}, there is
a characteristic abelian subgroup $A$ in $H$ whose order is coprime to $p$ and whose
index is at most $B^{r+s}\cdot |H_p|^{e(r+s)+1}$, where $H_p$ is a $p$-Sylow subgroup of~$H$.
Therefore, $A$ is a normal abelian subgroup of $G$ of index at most
$$
\left(\bar{B}\cdot |\bar{G}_p|^{\bar{e}}\right)\cdot \left(B^{r+s}\cdot |H_p|^{e(r+s)+1}\right)\le \bar{B}B^{r+s}\cdot |G_p|^{\bar{e}+e(r+s)+1},
$$
where $G_p$ is a $p$-Sylow subgroup of~$G$.
\end{proof}

We will use the following general fact.

\begin{lemma}\label{lemma:Darafsheh}
Let $p$ be a prime number, and let $m$ be a non-negative integer.
Let $F$ be a group containing a normal
subgroup $F'\cong (\ZZ/p\ZZ)^m$, and let $g$ be an element of $F$.
Then for some positive integer $t\le p^m-1$ the element~$g^t$ commutes with $F'$.
\end{lemma}

\begin{proof}
Let $L\subset F$ be the subgroup generated by $g$.
The action of $L$ on $F'$ defines a
homomorphism
$$
\chi\colon L\to \Aut(F')\cong\GL_m(\mathbf{F}_p).
$$
It is known that the order of any element in
$\GL_m(\mathbf{F}_p)$ does not exceed~\mbox{$p^m-1$}, see
e.g.~\cite[Corollary~2]{Darafsheh}.
Therefore, $g^t$ is contained in the kernel of $\chi$ for some $t\le p^m-1$, and the required
assertion follows.
\end{proof}

Lemma~\ref{lemma:Darafsheh} allows to obtain a version of Lemma~\ref{lemma:Heisenberg-estimate-1} that is applicable for a certain class of groups
without a bound on the number of generators.

\begin{lemma}\label{lemma:Heisenberg-estimate-2}
Let $p$ be a prime number, and let
$$
1\to F\to H\to \bar{H}\to 1
$$
be an exact sequence of finite groups such that $\bar{H}$ is abelian and is generated by~$r$ elements.
Suppose that the $p$-Sylow subgroup $F_p$ of $F$ is normal
in~$F$, and~\mbox{$F_p\cong(\ZZ/p\ZZ)^m$} for some non-negative integer $m$. Furthermore, suppose
that~\mbox{$|F|\le B\cdot |F_p|^e$} for some positive constants $B$ and $e$, and suppose that
$F$ is generated by $F_p$ and $s$ additional elements.
Then~$H$ contains a characteristic abelian subgroup of order coprime to $p$ and
index at most~\mbox{$B^{r+s+1}\cdot |H_p|^{e(r+s+1)+r+1}$}, where $H_p$ is a $p$-Sylow subgroup of~$H$.
\end{lemma}

\begin{proof}
Let us use the notation of the proof of Lemma~\ref{lemma:Heisenberg-estimate-1}.
We are going to estimate the index of the center $Z$ of $H$, and its  maximal subgroup~$Z'$ of order coprime to~$p$.

Let $x_1,\ldots, x_r$ be the elements of $H$ such that their images in $\bar{H}$ generate~$\bar{H}$,
and let $y_1,\ldots,y_s$ be the elements of $F$ such that $F_p$ and $y_1,\ldots,y_s$ generate $F$.
Set
$$
R=Z(x_1)\cap\ldots\cap Z(x_r)\cap Z(y_1)\cap\ldots\cap Z(y_s).
$$
Then
\begin{multline*}
[H:R]\le |K(x_1)|\cdot\ldots\cdot |K(x_r)|\cdot |K(y_1)|\cdot\ldots\cdot |K(y_s)| \\
\le |K|^{r+s} \le |F|^{r+s}\le B^{r+s}\cdot |F_p|^{e(r+s)}.
\end{multline*}

Let $\bar{R}$ be the image of $R$ in $\bar{H}$. Being a subgroup of an abelian group
generated by $r$ elements, $\bar{R}$ can be generated by $r$ elements as well.
Let $z_1,\ldots,z_r$ be the elements of $R$ such that their images in $\bar{H}$ generate
$\bar{R}$.  Then $z_1,\ldots,z_r$ normalize the group $F_p$ by Example~\ref{example:normal-Sylow-subgroup}.
According to Lemma~\ref{lemma:Darafsheh}, there exist positive integers $t_i\le p^m-1$,
$1\le i\le r$, such that the elements~$z_i^{t_i}$ commute with the subgroup $F_p$.
Therefore, the group $R'$ generated by~$z_i^{t_i}$,~\mbox{$1\le i\le r$}, is contained in the center $Z$ of $H$.
On the other hand, the image~$\bar{R}'$ of~$R'$ in $\bar{H}$ is a subgroup of index at most
$t_1\cdot\ldots\cdot t_r$ in $\bar{R}$, which implies that
$$
[R:R']\le t_1\cdot\ldots\cdot t_r\cdot |F|< p^{mr}\cdot B\cdot |F_p|^e=B\cdot |F_p|^{e+r}.
$$
Finally, as in the proof of Lemma~\ref{lemma:Heisenberg-estimate-1},
we have
\begin{multline*}
[H:Z']=[H:Z]\cdot [Z:Z']\le [H:R']\cdot [Z:Z']= [H:R]\cdot [R:R']\cdot [Z:Z']\\ \le
\left(B^{r+s}\cdot |F_p|^{e(r+s)}\right)\cdot \left(B\cdot |F_p|^{e+r}\right)\cdot |H_p|\le
B^{r+s+1}\cdot |H_p|^{e(r+s+1)+r+1}.
\end{multline*}
\end{proof}

\medskip
\textbf{Automorphisms of lattices.}
Given a prime number $\ell$, we denote by~$\ZZ_\ell$ the ring of
$\ell$-adic integers. The following assertion
is well known, and
is proved similarly to the classical theorem of H.\,Minkowski,
see \cite[\S1]{Minkowski}, or~\mbox{\cite[Lemma~1]{Serre-2007}}, or~\mbox{\cite[Theorem~9.9]{Serre-FiniteGroups}}.

\begin{lemma}
\label{lemma:Minkowski-l-adic}
Let $n$ be a positive integer, let $\ell$ be a prime, and let
$G$ be a finite subgroup of $\GL_n(\ZZ_\ell)$.
Then $G$ is isomorphic to a subgroup of $\GL_n(\ZZ/\ell\ZZ)$
if~\mbox{$\ell\neq 2$}, and to a subgroup of $\GL_n(\ZZ/4\ZZ)$
if $\ell=2$. In particular,
the group~\mbox{$\GL_n(\ZZ_\ell)$} has bounded finite subgroups.
\end{lemma}

\begin{proof}
First assume that $\ell\neq 2$.
Let
$$
\rho\colon \GL_n(\ZZ_\ell)\to \GL_n(\ZZ/\ell\ZZ)
$$
be the natural homomorphism. We claim that
its kernel does not contain non-trivial elements of finite order.
Indeed, denote by $\one$ the identity matrix in~\mbox{$\GL_n(\ZZ_\ell)$},
and suppose that 
$$
M=\one +\ell M'
$$
is an element of $\GL_n(\ZZ_\ell)$ such that $M^r=\one$ for some positive
integer $r$. Set
\begin{equation}\label{equation:log}
\log M=\ell M'-\frac{(\ell M')^2}{2}+\ldots+(-1)^{k-1}\frac{(\ell M')^k}{k}
+\ldots
\end{equation}
It is easy to see that the series on the right hand side
of~\eqref{equation:log} converges in~\mbox{$\GL_n(\ZZ_\ell)$}, and so
$\log M$ is a well defined element of $\GL_n(\ZZ_\ell)$; also, we see
that $\log M$ is divisible by $\ell$ in $\GL_n(\ZZ_\ell)$.
Now let $L$ be an arbitrary element divisible by $\ell$ in $\GL_n(\ZZ_\ell)$.
Set
\begin{equation}\label{equation:exp}
\exp L=\one + L+\ldots+\frac{L^k}{k!}+\ldots
\end{equation}
Note that the $\ell$-adic valuation of $L^k$ is at least $k$, while
the $\ell$-adic valuation of~$k!$ equals
$$
\left\lfloor\frac{k}{\ell}\right\rfloor+\left\lfloor\frac{k}{\ell^2}\right\rfloor+\ldots\le
k\left(\frac{1}{\ell}+\frac{1}{\ell^2}+\ldots\right)=\frac{k}{\ell-1}\le
\frac{k}{2}.
$$
Thus the series on the right hand side
of~\eqref{equation:exp} converges in~\mbox{$\GL_n(\ZZ_\ell)$},
and so~\mbox{$\exp L$} is a well defined element of $\GL_n(\ZZ_\ell)$.
We conclude that
$$
\exp \log M=M
$$
and
$$
r\log M=\log M^r=0.
$$
Hence $\log M=0$ and $M=\one$.

Now assume that $\ell=2$. Arguing as above, one shows that the kernel
of the natural homomorphism
$$
\rho\colon \GL_n(\ZZ_2)\to \GL_n(\ZZ/4\ZZ)
$$
does not contain non-trivial elements of finite order.
\end{proof}

Lemma~\ref{lemma:Minkowski-l-adic} allows to deduce more traditional versions of
Minkowski's theorem.

\begin{corollary}\label{corollary:Minkowski}
Let $\Lambda$ be a finitely generated abelian group. Then the group~\mbox{$\Aut(\Lambda)$}
has bounded finite subgroups.
\end{corollary}

\begin{proof}
There is an exact sequence of groups
$$
1\to\big(\Lambda_{\mathrm{tors}}\big)^{\times n}\times \Aut(\Lambda_{\mathrm{tors}})\to \Aut(\Lambda)\to\GL_n(\ZZ)\to 1,
$$
where $\Lambda_{\mathrm{tors}}$ is the torsion subgroup of $\Lambda$, and
$n$ is the rank of the free abelian group~\mbox{$\Lambda/\Lambda_{\mathrm{tors}}$}.
The group $\GL_n(\ZZ)$ is a subgroup
of $\GL_n(\ZZ_\ell)$ for any prime~$\ell$. Thus, it has bounded finite subgroups by Lemma~\ref{lemma:Minkowski-l-adic}.
On the other hand, the group $\Aut(\Lambda_{\mathrm{tors}})$ is finite, and the required assertion follows.
\end{proof}

\begin{corollary}[{cf. \cite[Theorem~F]{Feit}}]
For every positive integer~$n$, the group~\mbox{$\GL_n(\QQ)$} has bounded finite subgroups.
\end{corollary}

\begin{proof}
If $G$ is a finite subgroup of $\GL_n(\QQ)$, then it
acts by automorphisms of some sublattice $\ZZ^n\subset\QQ^n$. This means that~$G$ is isomorphic to
a subgroup of~\mbox{$\GL_n(\ZZ)$}, and thus $\GL_n(\QQ)$ has bounded finite by Corollary~\ref{corollary:Minkowski}.
\end{proof}

\medskip
\textbf{Projective general linear groups.}
We conclude this section by an easy consequence
of Theorem~\ref{theorem:BrauerFeit}.

\begin{corollary}\label{corollary:PGL}
Let $n$ be a positive integer. Then there
exists a constant~\mbox{$J_{\PGL}(n)$} such that every finite subgroup
$G$ of~\mbox{$\PGL_n(\Bbbk)$}, where $\Bbbk$
is an arbitrary field of characteristic $p>0$,
contains a normal abelian subgroup of order coprime
to $p$ and index at most~\mbox{$J_{\PGL}(n)\cdot |G_{p}|^{3}$},
where $G_{p}$ is a $p$-Sylow subgroup of $G$.
\end{corollary}

\begin{proof}
The adjoint representation embeds the group
$\PGL_n(\Bbbk)$ into~\mbox{$\GL_{n^2}(\Bbbk)$},
so Theorem~\ref{theorem:BrauerFeit} applies.
\end{proof}

\section{Automorphism groups}
\label{section:Aut}

In this section we recall the basic concepts and facts about
varieties and their automorphisms.

\medskip
\textbf{General settings.}
The following terminology is standard (although not \emph{universally}
accepted); see e.g. \cite[\S\,AG.12, \S\,I.1]{Borel}, and cf. \cite[Definition~2.1.5]{Brion-long}.

\begin{definition}
An algebraic group over a field $\Bbbk$ is a geometrically reduced group scheme of finite type over $\Bbbk$.
\end{definition}

In other words, an algebraic group is a variety with a structure of a group scheme.
Note that any algebraic group is smooth, see for instance~\mbox{\cite[Proposition~2.1.12]{Brion-long}}.

If $X$ is an arbitrary algebraic variety over a field $\Bbbk$, then its automorphisms form a group, which
we denote by $\Aut(X)$. If $K\supset\Bbbk$ is a field extension, and~$X_K$ is the extension of scalars of $X$ to $K$,
then every automorphism of $X$ defines an automorphism of $X_K$; in other words, one
has
$$
\Aut(X)\subset\Aut(X_K).
$$
If $X$ is irreducible,
then one can consider its birational automorphism group~\mbox{$\Bir(X)$};
one has a natural embedding $\Aut(X)\subset\Bir(X)$.
If $X_K$ is still irreducible,
then $\Bir(X)\subset\Bir(X_K)$.
In this paper, we will not be interested in any additional structures on the group
$\Bir(X)$; the reader can find a discussion of these
in some particular cases over fields of characteristic zero
in \cite{Hanamura1} and~\cite{Hanamura2} (see also \cite{BlancFurter}).
However, we will need some structure related to the automorphism group.

Let $X$ be a projective variety over a field $\Bbbk$.
Then the automorphism functor of
$X$ is represented by the automorphism group scheme $\AAut_X$
which is locally of finite type, see e.g. \cite[Theorem~7.1.1]{Brion-long}.
The automorphism
group~\mbox{$\Aut(X)$} is just the group of $\Bbbk$-points
of $\AAut_X$, that is,
$$
\Aut(X)=\AAut_X(\Bbbk).
$$
Let $\AAut^0_X$ be the neutral component of $\AAut_X$; then $\AAut^0_X$ is a connected group scheme of finite type
over $\Bbbk$, but it may be non-reduced,
even if $\Bbbk$ is algebraically closed and $X$ is smooth,
see e.g. \cite[Example~7.1.5]{Brion-long}.
Let $\AAut^0_{X,\red}$ be the maximal reduced subscheme of $\AAut^0_X$.
If the field $\Bbbk$ is perfect,
then $\AAut^0_{X,\red}$ is a group scheme, see \cite[\S2.5]{Brion-long}
(note that over a non-perfect field
the maximal reduced subscheme of a group scheme is not necessarily
a group scheme itself, see  \cite[Example~2.5.3]{Brion-long}).
Furthermore, in this case the fact that  $\AAut^0_{X,\red}$
is reduced implies that it is geometrically reduced,
and thus $\AAut^0_{X,\red}$ is an algebraic group;
in particular, this means that~\mbox{$\AAut^0_{X,\red}$} is smooth.
We set
$$
\Aut^0(X)=\AAut^0_{X,\red}(\Bbbk)=\AAut^0_X(\Bbbk).
$$
Note that $\Aut^0(X)$ always has a group structure,
regardless of whether the field~$\Bbbk$ is perfect or not.
Anyway, in this paper we will need to deal with the group~\mbox{$\Aut^0(X)$} and
the group scheme $\AAut^0_X$ only in the case when the base field
is algebraically closed.

The reader is referred to the surveys \cite[\S7.1]{Brion-long} and \cite[\S2.1]{Brion-short} for more details
on automorphism groups and automorphism group schemes.

\medskip
\textbf{Resolution of singularities and regularization of birational maps.}
Resolution of singularities is available in arbitrary dimension over fields
of characteristic $0$, and in small dimensions over fields of positive
characteristic. We will need it in the classical case of surfaces.

\begin{theorem}[{see \cite[Theorem]{Lipman-Desingularization} and \cite[Corollary~27.3]{Lipman-Rational}}]
\label{theorem:resolution-of-singularities}
Let $S$ be a geometrically irreducible algebraic surface
over a field $\Bbbk$.
Then there exists a minimal resolution of singularities of $S$.
More precisely, there exists a regular
surface $\tilde{S}$ over $\Bbbk$ with a proper birational
morphism~\mbox{$\pi\colon \tilde{S}\to S$} such that any proper birational morphism from
a regular projective surface to~$S$ factors through~$\pi$. In particular, if the field $\Bbbk$ is perfect, then
the surface $\tilde{S}$ is smooth.
\end{theorem}

\begin{remark}\label{remark:projective}
If in the notation of Theorem~\ref{theorem:resolution-of-singularities} the surface~$S$ is projective, then the surface~$\tilde{S}$
is projective as well. Indeed,~$\tilde{S}$ is complete and regular. Hence~$\tilde{S}_{\bar{\Bbbk}}$ is projective according
to~\mbox{\cite[Corollary~IV.2.4]{Kleiman}},
which implies that~$\tilde{S}$ is projective, see e.g.~\mbox{\cite[Corollaire~9.1.5]{EGA4-3}}.
\end{remark}

In particular, Theorem~\ref{theorem:resolution-of-singularities} and Remark~\ref{remark:projective} tell us that
every geometrically irreducible algebraic surface has a regular projective birational model.

\begin{corollary}\label{corollary:smooth-projective-model}
Let $S$ be a geometrically irreducible algebraic surface
over a field $\Bbbk$. Then there exists a regular
projective surface $\tilde{S}$ over $\Bbbk$ birational to~$S$.
In particular, if the field $\Bbbk$ is perfect, then
the surface $\tilde{S}$ is smooth.
\end{corollary}

\begin{proof}
Replace $S$ by its affine open subset, then replace the latter by a projective completion,
and take a resolution of singularities.
\end{proof}

The factorization property provided by Theorem~\ref{theorem:resolution-of-singularities}
implies that the minimal resolution of singularities behaves well with respect to the automorphism group.

\begin{corollary}\label{corollary:Aut-lifts-to-resolution}
Let $S$ be a geometrically irreducible algebraic surface, and let~\mbox{$\pi\colon \tilde{S}\to S$} be the minimal resolution of singularities.
Then there is an action of the group $\Aut(S)$ on $\tilde{S}$ such that the morphism $\pi$ is $\Aut(S)$-equivariant.
\end{corollary}

The following version of Corollary~\ref{corollary:smooth-projective-model} taking into account
a birational action of a finite group is classical and widely used, at least in the case of zero
characteristic, see e.g. \cite[Lemma~3.5]{DI}
or~\mbox{\cite[Lemma-Definition~3.1]{ProkhorovShramov-Bir}},
and cf.~\cite{Weil}, \cite{Brion-regularization}.
We recall its proof for the convenience of the reader.

\begin{lemma}\label{lemma:regularization}
Let $S$ be a geometrically irreducible algebraic surface over a field~$\Bbbk$, and let $G\subset\Bir(S)$ be a finite group. Then there
exists a regular projective surface $\tilde{S}$ with a biregular action of $G$ and a $G$-equivariant
birational map~\mbox{$\tilde{S}\dasharrow S$}. In particular, if the field $\Bbbk$ is perfect, then
the surface $\tilde{S}$ is smooth.
\end{lemma}

\begin{proof}
%Replacing $S$ by its normalization, we may assume that it is normal.
%There exists a $G$-invariant affine open subset $U$ in $S$ such that $G$ acts on $S$ biregularly.
%Thus the quotient
%$$
%V=U/G=\Spec \Bbbk[U]^G
%$$
%is also affine and normal, see e.g.~\mbox{\cite[\S\,III.12]{Serre-AGCF}}.
%There exists a normal projective completion $\hat{V}\supset V$.
%
%The function field $\Bbbk(U)=\Bbbk(S)$ is a finite extension of $\Bbbk(V)=\Bbbk(\hat{V})$.
%Let~\mbox{$\nu\colon\hat{S}\to \hat{V}$} be the normalization of $\hat{V}$  in~\mbox{$\Bbbk(U)$},
%see~\mbox{\cite[\S\,I.1]{Milne}}.
%Then~$\nu$ is a finite morphism, see \cite[Proposition~I.1.1]{Milne}
%and~\mbox{\cite[Remark~I.1.2]{Milne}}. Thus, $\nu$ is projective, see e.g.~\mbox{\cite[\href{https://stacks.math.columbia.edu/tag/0B3I}{Tag 0B3I}]{Stack}}.
%This means that the surface~$\hat{S}$ is projective.
%Note that the field~$\Bbbk(U)$ is acted on by the group $G$, and the integral closure of any $G$-invariant subalgebra
%of~$\Bbbk(U)$ in~$\Bbbk(U)$ is again $G$-invariant. By construction of $\nu$, this means that $\hat{S}$
%is covered by affine charts with a regular action of $G$; they glue to give a regular
%action of~$G$ on~$\hat{S}$.
%Since the integral closure of $\Bbbk[V]=\Bbbk[U]^G$ in $\Bbbk(U)$ is $\Bbbk[U]$, we see that one of these affine
%charts is isomorphic to $U$ with the initial regular action of~$G$.
%This provides a $G$-equivariant birational map $\zeta\colon \hat{S}\dashrightarrow S$.
Let $\hat{V}$ be a normal projective model of the field of invariants~\mbox{$\Bbbk(S)^G$},
and let~$\hat{S}$ be the normalization of~$\hat{V}$ in the field~$\Bbbk(S)$.
Then there is a regular action of $G$ on~$\hat{S}$ and a $G$-equivariant birational map~\mbox{$\zeta\colon \hat{S}\dashrightarrow S$}.

Let $\pi\colon \tilde{S}\to\hat{S}$ be the minimal resolution
of singularities provided by Theorem~\ref{theorem:resolution-of-singularities}.
Then $\tilde{S}$ is a regular geometrically irreducible surface; also, $\tilde{S}$ is projective
by Remark~\ref{remark:projective}.
According to Corollary~\ref{corollary:Aut-lifts-to-resolution}, the action of $G$ lifts to $\tilde{S}$ so that
the morphism $\pi$ is $G$-equivariant. Therefore, we obtain a $G$-equivariant
birational map~\mbox{$\zeta\circ\pi\colon\tilde{S}\dasharrow S$}.
\end{proof}

\medskip
\textbf{Stabilizer of a point.}
Let us recall an auxiliary result on fixed points of automorphisms of finite order
which is well known and widely used (especially in characteristic~$0$).
We provide its proof to be self-contained.
Given an algebraic variety $X$ over a field $\Bbbk$ and a $\Bbbk$-point~\mbox{$P\in X$}, denote by~$T_P(X)$ the Zariski tangent space to $X$ at~$P$.

\begin{theorem}\label{theorem:fixed-point}
Let $X$ be an irreducible algebraic variety
over a field $\Bbbk$ of characteristic~$p$.
Let $G$ be a finite group acting on
$X$ with a fixed $\Bbbk$-point~$P$. Suppose that $|G|$ is not divisible by~$p$. Then the natural representation
$$
d\colon G\to\GL\big(T_P(X)\big)
$$
is an embedding.
\end{theorem}

\begin{proof}
Suppose that $d$ is not an embedding.
Replacing $G$ by the kernel of $d$, we may assume that
$G$ acts trivially on $T_P(X)$. So $G$ acts trivially on
$$
\mathfrak{m}_P/\mathfrak{m}_P^2\cong T_P(X)^\vee,
$$
where $\mathfrak{m}_P\subset \mathcal{O}_P$ is the maximal ideal in the local ring of the point $P$ on~$X$.
The quotient morphism $\mathcal{O}_P\to \mathcal{O}_P/\mathfrak{m}_P\cong\Bbbk$ admits a natural section, which gives a $G$-invariant
decomposition~\mbox{$\mathcal{O}_P\cong \Bbbk\oplus \mathfrak{m}_P$} into a direct sum of vector subspaces.

Consider the $G$-invariant filtration
$$
\mathfrak{m}_P\supset \mathfrak{m}_P^2\supset \mathfrak{m}_P^3\supset  \ldots
$$
Recall that  $\mathfrak{m}_P$ is generated by elements of degree $1$, i.e. generated by a collection of elements
whose images form a basis in~\mbox{$\mathfrak{m}_P/\mathfrak{m}_P^2$}. Hence $G$
acts trivially on~\mbox{$\mathfrak{m}_P^n/\mathfrak{m}_P^{n+1}$} for every positive integer~$n$.
Since the order of $G$ is not divisible by~$p$, every representation of $G$ is completely reducible.
Therefore, we have an isomorphism of $G$-representations
$$
\mathfrak{m}_P\cong\bigoplus\limits_{n=1}^\infty \mathfrak{m}_P^n/\mathfrak{m}_P^{n+1}.
$$
Thus, the action of $G$ on $\mathfrak{m}_P$ and $\mathcal{O}_P$ is trivial.

Let $U\subset X$ be an affine open subset containing the point~$P$. Then~\mbox{$U'=\sigma(U)$}
is also an affine open subset of $X$ containing~$P$. Let $R$ and~$R'$
be the coordinate rings of $U$ and $U'$, respectively. Then
$R$ and~$R'$ are subalgebras of~$\mathcal{O}_P$, and one has $\sigma^*R'=R$. Since the action of
$\sigma$ on $\mathcal{O}_P$ is trivial, we conclude that~\mbox{$R=R'$}, and $\sigma$ acts trivially on $R$.
This means that $U$ is $\sigma$-invariant, and~$\sigma$ acts trivially on~$U$.
Finally, since $X$ is irreducible,
$U$ is dense in~$X$; hence~$\sigma$ acts trivially on the whole~$X$.
\end{proof}

\section{Group of connected components}
\label{section:MZ}

In this section we recall the following assertion
established in~\mbox{\cite[Lemma~2.5]{MZ}}
(cf. the proof of~\mbox{\cite[Theorem~1.9]{Hu}}).

\begin{lemma}\label{lemma:MZ}
Let $X$ be a (possibly reducible) projective variety.
Then the group~\mbox{$\Aut(X)/\Aut^0(X)$} has bounded finite subgroups.
\end{lemma}

We provide the proof of Lemma~\ref{lemma:MZ} for the reader's convenience.
The argument below is mostly taken from~\mbox{\cite[Remark~2.6]{MZ}}.
Let us start with a simple observation.

\begin{lemma}\label{lemma:Aut-Aut0-alg-closed}
Let $X$ be a projective variety over a field $\Bbbk$.
Then there is an embedding
$$
\theta\colon \Aut(X)/\Aut^0(X)\hookrightarrow\Aut(X_{\bar{\Bbbk}})/\Aut^0(X_{\bar{\Bbbk}}).
$$
\end{lemma}

\begin{proof}
Let $\hat{g}$ be an element of the quotient group $\Aut(X)/\Aut^0(X)$, and let $g$ be its preimage in $\Aut(X)$.
Thus, $g$ is a $\Bbbk$-point of the group scheme $\AAut_X$.
Considering it as a $\bar{\Bbbk}$-point of $\AAut_X$, we obtain the
map $\theta$.

To show that $\theta$ is injective, suppose that~$\hat{g}$ is a non-trivial element,
so that~$g$ is not contained in $\Aut^0(X)$.
In other words, the corresponding $\Bbbk$-point of~$\AAut_X$ is not contained in the
closed subgroup $\AAut_X^0$, which is also defined over~$\Bbbk$. Therefore, $g$ is not contained in the set of $\bar{\Bbbk}$-points of~\mbox{$\AAut_X^0$},
which means that it defines a non-trivial element
of the quotient group~\mbox{$\Aut(X_{\bar{\Bbbk}})/\Aut^0(X_{\bar{\Bbbk}})$} as well.
\end{proof}

To prove Lemma~\ref{lemma:MZ}, we will use the following standard
fact.

\begin{theorem}[{\cite[Theorem~II.2.1]{Kleiman}}]
\label{theorem:numerically-trivial}
Let $X$ be a projective variety over an algebraically closed field.
Let $\mathcal{F}$ be a coherent sheaf, and let
$\mathcal{G}$ be a numerically trivial line bundle on~$X$.
Then the Euler characteristic of~$\mathcal{F}$ equals
the Euler characteristic of~$\mathcal{F}\otimes\mathcal{G}$.
\end{theorem}

Given a projective variety $X$ over an algebraically closed field $\Bbbk$,
let $\NS(X)$ denote Neron--Severi group of line bundles on $X$ modulo
algebraic equivalence.
Recall from \cite[Th\'eor\`eme~5.1]{Kleiman-NS} that
$\NS(X)$ is a finitely generated abelian group.
If $L$ is a line bundle on $X$, we define
the group scheme $\AAut_{X;[L]}$ as the stabilizer in~$\AAut_X$ of the class of~$L$ in
$\NS(X)$. Obviously, one has~\mbox{$\AAut^0_X\subset\AAut_{X;[L]}$}; this means that~$\AAut^0_X$
is the neutral component of~\mbox{$\AAut_{X;[L]}$}.

\begin{lemma}\label{lemma:Aut-X-L}
Let $X$ be a projective variety over an algebraically closed field, and let $L$ be an ample line bundle on~$X$.
Then $\AAut_{X;[L]}$ is a group scheme of finite type.
\end{lemma}

\begin{proof}
One can identify $\AAut_X$ with an open subscheme of the Hilbert scheme~\mbox{$\Hilb(X\times X)$}
by associating with each automorphism $f$ its graph~\mbox{$\Gamma_f\subset X\times X$};
see \cite[Theorem~5.23]{FGA-explained}, and the exercise after this theorem.
Set
$$
L_{X\times X}=p_1^*L\otimes p_2^*L,
$$
where $p_1, p_2\colon X\times X\to X$ are the projections
to the first and the second factor, respectively.
Then $L_{X\times X}$ is an ample line bundle on~\mbox{$X\times X$}.
Its restriction to~$\Gamma_f$ (identified with~$X$ via the projection~$p_1$) is the line bundle
$$
L_f\cong L\otimes f^*L.
$$
If~$f$ preserves the class $[L]\in \NS(X)$, then~$L_f$ is algebraically equivalent to~$L^2$;
in particular, these line bundles are numerically equivalent.
Therefore,  by Theorem~\ref{theorem:numerically-trivial}
the Hilbert polynomial of $\Gamma_f$ (with respect to the ample line bundle~$L_{X\times X}$ on~$X$)
is independent of $f$ provided that~\mbox{$f\in \AAut_{X;[L]}$}.

Let $P$ be this polynomial. Then $\AAut_{X;[L]}$ is contained in the Hilbert scheme~\mbox{$\Hilb_P(X\times X)$}. Recall that $\Hilb_P(X\times X)$ is projective,
see for instance~\mbox{\cite[Theorem~I.1.4]{Kollar-RatCurves}}.
Set
$$
\AAut_X^P=\AAut_X\cap \Hilb_P(X\times X)\subset \Hilb(X\times X).
$$
Then $\AAut_X^P$ is an open subscheme of $\Hilb_P(X\times X)$, and so it is quasi-projective.
Furthermore, $\AAut_{X;[L]}$ is a closed subscheme of~$\AAut_X^P$.
Hence, $\AAut_{X;[L]}$ is quasi-projective as well, thus it is of finite type.
\end{proof}

Denote by $\Aut_{[L]}(X)$ the group of $\Bbbk$-points of~\mbox{$\AAut_{X;[L]}$}.
The following assertion is implied by Lemma~\ref{lemma:Aut-X-L}.

\begin{corollary}\label{corollary:Aut-X-L}
Let $X$ be a projective variety over an algebraically closed field~$\Bbbk$, and let~$L$ be an ample line bundle on $X$.
Then $\Aut_{[L]}(X)/\Aut^0(X)$ is a finite group.
\end{corollary}

\begin{proof}
It follows from Lemma~\ref{lemma:Aut-X-L} that
$\AAut_{X;[L]}/\AAut^0_{X}$ is a finite group scheme.
Therefore, its group of $\Bbbk$-points
$$
\big(\AAut_{X;[L]}/\AAut^0_X\big)(\Bbbk)\cong \Aut_{[L]}(X)/\Aut^0(X).
$$
is finite.
\end{proof}

Now we prove Lemma~\ref{lemma:MZ}.

\begin{proof}[{Proof of Lemma~\ref{lemma:MZ}}]
By Lemma~\ref{lemma:Aut-Aut0-alg-closed}, we may assume that $X$ is defined over an algebraically closed field.
Choose an ample line bundle $L$ on $X$.
Let $G$ be a finite subgroup of the group~\mbox{$\Aut(X)/\Aut^0(X)$}. Since~\mbox{$\Aut^0(X)$} acts trivially on the Neron--Severi group~\mbox{$\NS(X)$},
there is a natural action of $G$ on~\mbox{$\NS(X)$}.
One has an exact sequence of groups
$$
1\longrightarrow K\longrightarrow G \longrightarrow \bar{G}\longrightarrow 1,
$$
where $\bar{G}$ is the image of the representation of $G$ in $\NS(X)$, and $K$ is the kernel of this representation. Since $\NS(X)$ is a finitely generated
abelian group, by Corollary~\ref{corollary:Minkowski}
there is a constant~\mbox{$M=M(X)$} such that every finite subgroup of $\Aut(\NS(X))$ has order at most $M$;
in particular, one has $|\bar{G}|\leqslant M$. On the other hand, the group~$K$ preserves the ample divisor class $[L]\in \NS(X)$, and hence
$$
K\subset \Aut_{[L]}(X)/\Aut^0(X).
$$
By Corollary~\ref{corollary:Aut-X-L}
the group $\Aut_{[L]}(X)/\Aut^0(X)$ is finite. Therefore, we have
$$
|G|=|\bar{G}|\cdot |K|\leqslant M\cdot \left|\Aut_{[L]}(X)/\Aut^0(X)\right|.
$$
\end{proof}

To conclude this section, let us make the following remark.
For every projective variety~$X$ defined over a field~$\Bbbk$,
there exists a natural embedding of groups
$$
\Aut(X)/\Aut^0(X)\hookrightarrow \big(\AAut_X/\AAut^0_X\big)(\Bbbk).
$$
However, over an algebraically non-closed field this embedding may fail to be an isomorphism.

\begin{example}
Let $\Bbbk=\QQ(\omega)$, where $\omega$ is a non-trivial cubic root of unity.
Let~$E$ be the curve given in $\PP^2$ with homogeneous coordinates~$x$, $y$, and $z$
by equation
$$
y^2z=x^3+z^3,
$$
and choose $(0:1:0)$ to be the marked point on $E$.
Then $E$ is an elliptic curve such that
$$
\AAut_{E}\cong\AAut^0_{E}\rtimes\ZZ/6\ZZ,
$$
and all six $\bar{\Bbbk}$-points of the group scheme
$\AAut_E/\AAut^0_E$ are defined over $\Bbbk$.
Note that all the $2$-torsion $\bar{\Bbbk}$-points of $\AAut^0_E$ are also defined over $\Bbbk$;
let $\upsilon$ be one of the non-trivial $2$-torsion $\Bbbk$-points of $\AAut^0_E$.
Choose an auxiliary quadratic extension~$\KK$ of~$\Bbbk$, say,~\mbox{$\KK=\Bbbk(\sqrt{7})$}.
Let $\zeta$ be the $1$-cocycle corresponding to the homomorphism
$\Gal(\KK/\Bbbk)\to\Aut(E)$ that sends
the generator of $\Gal(\KK/\Bbbk)\cong\ZZ/2\ZZ$ to~\mbox{$\upsilon\in\Aut(E)$}.
Let $E'$ be the twist of~$E$ by~$\zeta$.
Then $E'$ is a smooth geometrically irreducible projective curve of genus $1$ over~$\Bbbk$.
One can check that~\mbox{$E'(\Bbbk)=\varnothing$}.
Moreover, the Jacobian $J(E')$ is the twist of the Jacobian~\mbox{$J(E)$} by the same cocycle $\zeta$.
However, $\upsilon$ acts trivially on~\mbox{$J(E)$}, and thus~\mbox{$J(E')\cong J(E)$}. Since $E$ has a $\Bbbk$-point, we also
have an isomorphism~\mbox{$J(E)\cong E$}.

Now observe that the group scheme $\AAut_{E'}$ acts on
the curve $E\cong J(E')$, and~\mbox{$\AAut^0_{E'}$} is contained in the kernel of this action.
This gives rise to a homomorphism
$$
\AAut_{E'}/\AAut^0_{E'}\to\AAut_E/\AAut^0_E
$$
that becomes an isomorphism after the extension of scalars to $\bar{\Bbbk}$.
Thus, this homomorphism is actually an isomorphism.
Since~\mbox{$\AAut_E/\AAut^0_E$} has six $\Bbbk$-points, we conclude that~\mbox{$\AAut_{E'}/\AAut^0_{E'}$} has six $\Bbbk$-points as well.
On the other hand, the group~\mbox{$\Aut(E')/\Aut^0(E')$} does not contain elements of order~$6$. Indeed,
it follows from Lefschetz fixed point formula that
any preimage of such an element in~\mbox{$\Aut(E')$} would have a unique fixed point on~\mbox{$E'_{\bar{\Bbbk}}\cong E_{\bar{\Bbbk}}$},
and so this point would be defined over~$\Bbbk$.
\end{example}

\section{Minimal models}
\label{section:surfaces}

In this section, we recall the notion of the Kodaira dimension, and discuss some facts
concerning classification of surfaces
and the Minimal Model Program in dimension~$2$.

\medskip
\textbf{Kodaira dimension.}
One of the most important birational invariants
of projective varieties is the Kodaira dimension.
Given a smooth geometrically irreducible
projective variety, we will denote by $\omega_X$ the canonical sheaf on~$X$,
and by $K_X$ the canonical class of $X$, i.e. the class of $\omega_X$
in~\mbox{$\Pic(X)$}.

\begin{definition}[{\cite[Definition 5.6]{Badescu}}]
Let $X$ be a smooth irreducible
projective variety  over an algebraically closed
field.
The Kodaira dimension $\kappa(X)$ of $X$ is defined as follows:
$$
\kappa(X)=\left\{
\begin{array}{ll}
  \mathrm{tr.deg} \left(\bigoplus\limits_{n=0}^{\infty} H^0(X,\omega_X^{\otimes n})\right)-1,&
  \text{if }
  \mathrm{tr.deg} \left(\bigoplus\limits_{n=0}^{\infty} H^0(X,\omega_X^{\otimes n})\right)>0;\\
  -\infty,&\text{if }H^0(X,\omega_X^{\otimes n})=0 \text{ for all }n\ge 1.
\end{array}\right.
$$
\end{definition}

\begin{remark}
Alternatively, one can define $\kappa(X)$ as the maximal
dimension of an image of $X$ with respect to the rational map
given by the pluricanonical linear systems $|mK_X|$ for all
$m\ge 1$, see \cite[\S10.5]{Iitaka}.
\end{remark}

For an arbitrary field $\Bbbk$, we define the Kodaira dimension
$\kappa(X)$ of a smooth geometrically irreducible projective
variety $X$ as the Kodaira dimension of~$X_{\bar{\Bbbk}}$.
One has
$$
-\infty\le\kappa(X)\le \dim X,
$$
see e.g. \cite[Lemma~5.5]{Badescu}.
Kodaira dimension is a birational invariant
for smooth geometrically irreducible projective varieties,
see \cite[\S10.5]{Iitaka}. Moreover, the following assertion holds.

\begin{lemma}[{see \cite[Corollary~IV.1.11]{Kollar-RatCurves}}]
\label{lemma:ruled-kappa}
Let $X$ be a smooth geometrically irreducible projective variety. Suppose that~$X$ is birational to $Y\times\PP^1$, where
$Y$ is a (possibly singular) algebraic variety.
Then~\mbox{$\kappa(X)=-\infty$}.
\end{lemma}

%\begin{proof}
%Denote $n=\dim X$.
%Let $U\subset Y$ be the smooth locus of $Y$.
%Then $U\times \PP^1$ is birational to~$X$.
%Let $\varphi\colon U\times \PP^1\dashrightarrow X$ be the birational map.
%Since $U\times \PP^1$ is smooth,
%the indeterminacy locus $Z\subset U\times \PP^1$
%of $\varphi$ has codimension at least~$2$. Thus
%$$
%\dim Z\leqslant n-2<n-1=\dim U.
%$$
%So the map $p_1\colon Z\rightarrow U$ is not surjective, where $p_1\colon U\times \PP^1\rightarrow U$ is the projection to the first factor.
%Therefore, there exists a nonempty open subset $V\subset U$ such that $V\times \PP^1$ is disjoint from $Z$.
%In other words, $\varphi\colon V\times \PP^1\rightarrow X$ is a birational morphism.
%
%Since $X$ is smooth, the cotangent sheaf $\Omega^1_X$ is locally free. Thus we have an embedding
%$$
%\varphi^*\Omega^1_X\hookrightarrow \Omega^1_{V\times \PP^1}.
%$$
%Since $V$ is also smooth, the pull back via $\varphi$ commutes with taking the exterior power,
%and so we get an embedding
%$$
%\varphi^*\omega_X\cong\varphi^*\left(\wedge^{n}\Omega^1_X\right)\cong\wedge^{n}\varphi^*\Omega^1_X\hookrightarrow \wedge^{n}\Omega^1_{V\times \PP^1}\cong\omega_{V\times \PP^1}.
%$$
%Therefore, we have
%$$
%H^0(X,\omega_X^{\otimes m})\hookrightarrow H^0(V\times \PP^1,\varphi^*\omega_X^{\otimes m})\hookrightarrow H^0(V\times \PP^1,\omega_{V\times \PP^1}^{\otimes m})
%$$
%for every positive integer $m$.
%On the other hand,
%one has~\mbox{$H^0(V\times \PP^1,\omega_{V\times \PP^1}^{\otimes m})=0$}
%by K\"unneth formula.
%Hence we see that $\kappa(X)=-\infty$.
%\end{proof}

In the case of surfaces (and in other cases when the resolution of singularities is available) one can extend the definition of
Kodaira dimension a little further. Namely, for an irreducible algebraic surface $S$ over an algebraically closed field,
the Kodaira dimension of $S$ is defined as the Kodaira dimension of (any) smooth projective birational model of $S$ (which exists
by Corollary~\ref{corollary:smooth-projective-model}). Note that due to birational invariance of Kodaira dimension this definition
does not depend on the choice of a birational model. If $S$ is  a geometrically irreducible algebraic surface
over an arbitrary field $\Bbbk$, we set~\mbox{$\kappa(S)=\kappa(S_{\bar{\Bbbk}})$.}

\medskip
\textbf{Minimal surfaces.}
Among all smooth projective surfaces, there is a special class
of the so-called minimal surfaces. They are the most important ones for studying
automorphism groups.

\begin{definition}[{see e.g. \cite[Definition 6.1]{Badescu}}]
Let $S$ be a smooth geometrically irreducible projective surface. We say that $S$ is \emph{minimal}
if every birational morphism $S\rightarrow Y$, where $Y$ is a
smooth projective surface, is an isomorphism.
\end{definition}

\begin{example}\label{example:minimal-0}
Let $S$ be a smooth geometrically irreducible projective surface over a field $\Bbbk$. Suppose that the canonical class $K_S$ is numerically trivial.
Then~$S$ is minimal.
Indeed, if $S$ is not minimal, then $S_{\bar{\Bbbk}}$ contains a
smooth rational curve with self-intersection~$-1$,
see e.g. \cite[Theorem~IV.3.4.5]{Shafarevich} or~\mbox{\cite[\S6]{Badescu}}.
On the other hand, computing the self-intersection by adjunction formula,
we see that such a curve cannot exist on~$S_{\bar{\Bbbk}}$.
\end{example}

Minimal surfaces are representatives of the birational equivalent classes
of all smooth geometrically irreducible projective surfaces. Indeed, since every birational morphism
between smooth projective surfaces over an algebraically closed field
is a composition of contractions of smooth rational curves with self-intersection~$-1$ (see \cite[Theorem~IV.3.4.5]{Shafarevich}),
the number of such consecutive contractions from a given surface $S$ is bounded by the Picard rank of~$S$.
Thus, we can replace every smooth geometrically irreducible projective surface $S$ over an arbitrary field by a minimal surface $S'$ such that there exists a birational
morphism from $S$ to~$S'$.

Birational automorphism groups of minimal surfaces of non-negative Kodaira dimension
are easy to study due to the following well known result.

\begin{lemma}[{see \cite[Corollary 10.22, Theorem 10.21]{Badescu}}]
\label{lemma:Bir-vs-Aut}
Let~$S$ be a minimal surface such that~\mbox{$\kappa(S)\ge 0$}.
Then $S$ is the unique minimal surface in its birational equivalence class, and~\mbox{$\Bir(S)=\Aut(S)$}.
\end{lemma}

\begin{remark}
In \cite{Badescu}, the proof of Lemma~\ref{lemma:Bir-vs-Aut} is given over
an algebraically closed field. However, the general case easily follows from this.
\end{remark}

Similarly to the case of characterstic zero, over fields of positive
characteristic
there exists a Kodaira--Enriques classification of minimal surfaces
due to E.\,Bombieri and D.\,Mumford \cite{Mumford-1969},
\cite{BM-II}, \cite{BM-III};
see also \cite{Badescu} and~\cite{Liedtke}.
We recall its part that will be used in this paper.
The definitions of the particular classes of surfaces can be found
for instance in~\mbox{\cite[\S6,\S7]{Liedtke}}.

\begin{theorem}
\label{theorem:KE}
Let $S$ be a smooth irreducible projective surface over an algebraically closed field.
The following assertions hold.
\begin{itemize}
\item[(i)]
If $\kappa(S)=-\infty$, then $S$ is birational
either to~$\PP^2$, or to~\mbox{$C\times\PP^1$},
where~$C$ is a (irreducible smooth projective) curve of positive genus.

\item[(ii)] If $\kappa(S)=0$ and $S$ is minimal, then $S$ is either a~$K3$ surface,
or an Enriques surface, or an abelian surface, or a hyperelliptic
surface, or a quasi-hyperelliptic surface.
\end{itemize}
\end{theorem}

\medskip
\textbf{$G$-minimal surfaces.}
There exists a version of the Minimal Model Program
that takes into account an action of a group.
Below we recall some of its implications in the case of geometrically rational surfaces.

\begin{definition}
A smooth geometrically irreducible
projective surface $S$ with an action of a
group $G$ is called $G$-\emph{minimal}
if every $G$-equivariant birational morphism
$S\rightarrow T$, where $T$ is a
smooth geometrically irreducible projective surface with an action of $G$,
is an isomorphism.
\end{definition}

Similarly to the case of the trivial group action, for every smooth geometrically irreducible projective
surface with an action of a group $G$ there is a $G$-equivariant birational
morphism to a $G$-minimal surface. Thus, it is interesting to
know the properties of $G$-minimal surfaces.

\begin{definition}\label{definition:conic-bundle}
Let $S$ be a smooth geometrically irreducible projective surface,
and let $\phi\colon S\to C$
be a surjective morphism to a smooth curve.
One says that~$\phi$ (or $S$) is a \emph{conic bundle},
if the fiber of $\phi$
over the scheme-theoretic generic point of $C$ is  smooth
and geometrically  irreducible, and the anticanonical line
bundle~$\omega^{-1}_S$ is very ample over $C$.
\end{definition}

\begin{remark}
Let $S$ be a smooth geometrically irreducible projective surface,
and let $\phi\colon S\to C$ be a conic bundle.
Denote by $S_\eta$ fiber of $\phi$ over the scheme-theoretic generic
point of $C$. Then $S_\eta$ is smooth and geometrically  irreducible.
Moreover, the anticanonical line
bundle~$\omega^{-1}_{S_\eta}$ is very ample.
This implies that~$S_\eta$ is a smooth conic over the field~$\Bbbk(C)$
of rational functions on~$C$.
\end{remark}

\begin{example}
Let $\Bbbk$ be an algebraically closed field of characteristic $2$.
Consider the surface $S\cong\PP^1\times\PP^1$, and let $\phi'\colon S\to\PP^1$
be the projection to the second
factor. Then $\phi'$ is a conic bundle. On the other hand,
let $\phi\colon S\to\PP^1$
be the composition
of~$\phi'$ with an inseparable double cover $\PP^1\to\PP^1$.
Then $\phi$ is \emph{not} a conic bundle. Indeed, its scheme-theoretic generic
fiber $S_{\eta}$ is not smooth
over the field~\mbox{$\Bbbk(\PP^1)$}. Moreover, for every point
$P$ of $S_{\eta}$, the scheme~$S_{\eta}$ is regular at~$P$,
but fails to be smooth
at this point, because $S_{\eta}$ is not geometrically reduced
(although it is reduced over~$\Bbbk$).
Note that $\phi$ satisfies the second requirement
of Definition~\ref{definition:conic-bundle},
i.e. the anticanonical sheaf~$\omega_S^{-1}$ is very ample
over~$\PP^1$; also, each geometric
fiber of~$\phi$ is isomorphic to a non-reduced
conic in~$\PP^2$.
\end{example}

Recall that a \emph{del Pezzo surface} is a smooth geometrically
irreducible projective
surface $S$  with ample anticanonical class. For the following result, we refer the
reader to~\mbox{\cite[Theorem~1G]{Iskovskikh}} or~\mbox{\cite[Theorem~2.7]{Mori}}
(cf. also~\mbox{\cite[Corollary~7.3]{Badescu}}).

\begin{theorem}
\label{theorem:Iskovskikh-G}
Let~$G$ be a finite group, and let~$S$ be a geometrically rational $G$-minimal surface.
Then~$S$ is either a del Pezzo surface, or a $G$-equivariant conic bundle over a smooth curve of genus zero.
\end{theorem}

If the base field is perfect, one can use Lemma~\ref{lemma:regularization} together with
Theorem~\ref{theorem:Iskovskikh-G}
to produce nice regularizations of birational actions of finite groups
on geometrically rational surfaces (note however that we will use this only over
algebraically closed fields in our paper).

\begin{theorem}
\label{theorem:MMP}
Let $S$ be a geometrically rational algebraic
surface over a perfect field.
Let $G$ be a finite subgroup of $\Bir(S)$.
Then there exists a
smooth geometrically irreducible projective surface $S'$ with a regular action of $G$ and a
$G$-equivariant map~\mbox{$S\dasharrow S'$}, such that $S'$ is either a del Pezzo surface, or a $G$-equivariant conic bundle over a
curve of genus zero.
\end{theorem}

\begin{proof}
First, there exists a regular projective surface $\tilde{S}$
with an action of~$G$
and a $G$-equivariant birational map $\tilde{S}\dasharrow S$, see Lemma~\ref{lemma:regularization}.
Since the base field is perfect, $\tilde{S}$ is actually smooth.
Thus there exists a
$G$-minimal surface~$S'$ that is $G$-equivariantly birational to~$\tilde{S}$ (and thus to~$S$).
Now the assertion about the geometrically rational case
follows from Theorem~\ref{theorem:Iskovskikh-G}.
\end{proof}

The next fact is well known.

\begin{theorem}
\label{theorem:del-Pezzo}
Let $S$ be a del Pezzo surface over a field~$\Bbbk$.
Then the linear system $|-3K_S|$ defines
an embedding $S\hookrightarrow\PP^N$, where $N\le 54$.
\end{theorem}

\begin{proof}
The divisor $-3K_S$ is very ample,
see~\mbox{\cite[Proposition~III.3.4.2]{Kollar-RatCurves}}.
Thus it defines an embedding $S\hookrightarrow\PP^N$,
where
$$
N=h^0\big(S,\mathcal{O}_S(-3K_S)\big)-1.
$$
On the other hand, one has $1\le K_S^2\le 9$,
see e.g.~\mbox{\cite[Theorem~IV.2.5]{Manin}}
or~\mbox{\cite[Exercise~III.3.9]{Kollar-RatCurves}}.
Hence
$$
h^0\big(S,\mathcal{O}_S(-3K_S)\big)-1=6K_S^2\le 54
$$
by \cite[Corollary~III.3.2.5]{Kollar-RatCurves}.
\end{proof}

\medskip
\textbf{Non-rational ruled surfaces.}
Birational automorphism groups of
non-rational surfaces covered by rational curves
have simpler structure than those of geometrically rational surfaces.

\begin{definition}\label{definition:fiberwise}
Given a morphism $\phi\colon X\to Y$ between varieties $X$ and~$Y$, and an automorphism (or a birational automorphism) $g$ of~$X$, we will
say that~$g$ is \emph{fiberwise with respect to $\phi$} if it maps
every point of $X$ to a point in the same fiber of $\phi$ (provided that $g$ is well defined at this point).
The action of a subgroup~$\Gamma$ of $\Aut(X)$ or $\Bir(X)$ is fiberwise with respect to $\phi$ if every element
of~$\Gamma$ is fiberwise with respect to~$\phi$.
\end{definition}

\begin{lemma}\label{lemma:ruled-Bir}
Let $\Bbbk$ be a field.
Let~$C$ be a smooth geometrically irreducible projective curve of positive genus
over $\Bbbk$, and set $S=C\times\PP^1$.
Then there is an exact sequence of groups
\begin{equation}\label{eq:Bir-exact-sequence}
1\to \Bir(S)_\phi\to \Bir(S)\to\Gamma,
\end{equation}
where $\Bir(S)_\phi\subset\PGL_2(\Bbbk(C))$
and $\Gamma\subset\Aut(C)$.
\end{lemma}

\begin{proof}
Consider the projection $\phi\colon S\to C$.
If $g\in\Bir(S)$, and $F\cong\PP^1$ is a general fiber of $\phi$,
then the map
$$
\phi\circ g\colon F\to C
$$
is either surjective,
or maps $F$ to a point. The former option is impossible,
because a rational curve cannot dominate a curve of positive genus,
see for instance \cite[Corollary~IV.2.4]{Hartshorne}
and \cite[Proposition~IV.2.5]{Hartshorne}.
This means that an image of $F$ under $g$ is again a fiber of $\phi$, so that
$\phi$ is equivariant with respect
to the whole group~\mbox{$\Bir(S)$}.

Hence there is an exact sequence~\eqref{eq:Bir-exact-sequence},
where the action of the group $\Bir(S)_\phi$
is fiberwise with respect to $\phi$, and $\Gamma$ is a subgroup of $\Bir(C)=\Aut(C)$. Thus,
the group~\mbox{$\Bir(S)_\phi$} is a subgroup of $\Bir(S_\eta)$, where
$S_\eta$ is the scheme-theoretic generic fiber of the map~$\phi$.
Since $S_\eta$ is isomorphic to the projective line over the field $\Bbbk(C)$, we have
$$
\Bir(S_\eta)=\Aut(S_\eta)\cong\PGL_2\big(\Bbbk(C)\big).
$$
\end{proof}

\section{Abelian varieties}
\label{section:abelian-varieties}

In this section we collect auxiliary facts about automorphism groups of abelian varieties.
The latter groups are usually infinite, but are rather easy to understand.
Note that here we regard
an abelian variety just as a variety, and do not take the
group structure into account. Also, if the base field is not algebraically closed, it is not necessary for us to assume that
an abelian variety has a point.

\begin{theorem}\label{theorem:abelian-variety}
Let $\Bbbk$ be an algebraically closed
field of characteristic $p\ge 0$, and let $\ell\neq p$
be a prime number.
Let $X$ be an $n$-dimensional abelian variety over $\Bbbk$.
Then
\begin{equation}\label{eq:abelian-variety}
\Aut(X)\cong X(\Bbbk)\rtimes\Aut(X;P),
\end{equation}
where
$X(\Bbbk)$ is the group of $\Bbbk$-points of $X$, and
$\Aut(X;P)$ is the stabilizer of a $\Bbbk$-point $P\in X$ in $\Aut(X)$.
Furthermore, $\Aut(X;P)$ is isomorphic to
a subgroup of~\mbox{$\GL_{2n}(\ZZ_\ell)$}.
\end{theorem}

\begin{proof}
The isomorphism~\eqref{eq:abelian-variety} is obvious.
Note that $\Aut(X;P)$ can be identified with the group of automorphisms
that preserve the group structure of~$X$.

Let $T_\ell(X)\cong\ZZ_\ell^{2n}$ be the Tate module of $X$, see \cite[\S18]{Mumford-AbelianVarieties}
for the definition. Let $\Hom(X,X)$ be the $\ZZ$-algebra of endomorphisms of the abelian variety~$X$.
According to \cite[\S19, Theorem~3]{Mumford-AbelianVarieties},
there is an injective homomorphism
$$
\Hom(X,X)\hookrightarrow
\Hom_{\ZZ_\ell}\big(T_\ell(X),T_\ell(X)\big).
$$
The group $\Aut(X;P)$ is the multiplicative group of invertible elements in~\mbox{$\Hom(X,X)$}.
Thus, it embeds into the group of invertible elements in~\mbox{$\Hom_{\ZZ_\ell}(T_\ell(X),T_\ell(X))$}, which is
$$
\Aut_{\ZZ_\ell}\big(T_\ell(X),T_\ell(X)\big)\cong\GL_{2n}(\ZZ_\ell).
$$
\end{proof}

\begin{remark}
In the notation of Theorem~\ref{theorem:abelian-variety}, if the characteristic of the field~$\Bbbk$
is zero, one can check that $\Aut(X;P)$
is a subgroup of $\GL_{2n}(\ZZ)$. It is well known that this also holds
if $X$ is a complex torus of dimension $n$, see e.g.~\mbox{\cite[Theorem~8.4]{ProkhorovShramov-CCS}}.
However, one can produce an example of an elliptic curve~$X$
over a field of characteristic $2$
such that the stabilizer of a point on~$X$ is a group of order~$24$,
see \cite[Exercise~A.1(b)]{Silverman}.
Such a group cannot be embedded into $\GL_2(\ZZ)$,
see~\mbox{\cite[\S1]{Tahara}}.
\end{remark}

The next two results are consequences of Theorem~\ref{theorem:abelian-variety}.
To formulate them, we set
\begin{equation}\label{eq:JA}
J_A(n)=
|\GL_{2n}(\ZZ/4\ZZ)|=
(2^{4n}-2^{2n})\cdot (2^{4n}-2^{2n+1})\cdot\ldots\cdot (2^{4n}-2^{4n-1}).
\end{equation}
Note that
$$
J_A(n)>
(3^{2n}-1)\cdot(3^{2n}-3)\cdot\ldots\cdot (3^{2n}-3^{2n-1})
=|\GL_{2n}(\ZZ/3\ZZ)|.
$$
for every positive integer $n$.

\begin{corollary}\label{corollary:abelian-variety-Jordan}
Let $n$ be a positive integer. For
every field $\Bbbk$, every abelian variety $X$ of dimension $n$ over $\Bbbk$,
every finite subgroup~\mbox{$G\subset \Aut(X)$}
contains a normal abelian subgroup of index at most~$J_A(n)$
that can be generated by at most $2n$ elements.
In particular, for every abelian variety $X$ the group $\Aut(X)$ is Jordan.
\end{corollary}

\begin{proof}
Let $X$ be an abelian variety over a field $\Bbbk$ of
characteristic $p\ge 0$,
and let~$G$ be a finite subgroup of $\Aut(X)$.
It is enough to consider the case when~$\Bbbk$ is algebraically closed.
Set $\ell=2$ if $p\neq 2$, and $\ell=3$ if $p=2$.
Note that the intersection $G_X$ of $G$ with
the subgroup~\mbox{$X(\Bbbk)\subset\Aut(X)$} of $\Bbbk$-points
of~$X$ is abelian and normal in~$G$.
Moreover, $G_X$ can be generated by at most~$2n$ elements, see e.g.~\mbox{\cite[\S15]{Mumford-AbelianVarieties}}.
On the other hand, by
Theorem~\ref{theorem:abelian-variety} the quotient~\mbox{$G/G_X$}
is isomorphic to a subgroup of $\GL_{2n}(\ZZ_\ell)$.
Thus, by Lemma~\ref{lemma:Minkowski-l-adic}
the order of~\mbox{$G/G_X$} does not exceed
$$
\max\left\{|\GL_{2n}(\ZZ/3\ZZ)|, |\GL_{2n}(\ZZ/4\ZZ)|\right\}=
J_A(n).
$$
\end{proof}

Applying Corollary~\ref{corollary:abelian-variety-Jordan}
together with Lemma~\ref{lemma:Jordan-vs-p-Jordan}, we obtain.

\begin{corollary}\label{corollary:abelian-variety-p-Jordan}
Let $p$ be a prime number, and let $n$ be a positive integer. For
every field $\Bbbk$, every abelian variety $X$ of dimension $n$ over $\Bbbk$,
every finite subgroup~\mbox{$G\subset \Aut(X)$}
contains a normal abelian subgroup $A$ such that $A$
can be generated by at most $2n$ elements, the order of $A$ is coprime to $p$, and the index
of $A$ in $G$ is at most $J_A(n)\cdot |G_p|$, where $G_p$ is a $p$-Sylow subgroup of $G$.
\end{corollary}

Arguing as in the proof of Corollary~\ref{corollary:abelian-variety-Jordan},
we can obtain restrictions on
stabilizers of points on abelian varieties.

\begin{corollary}\label{corollary:torus-stabilizer}
Let $n$ be a positive integer.
For every algebraically closed field~$\Bbbk$,
every abelian variety $X$ of dimension $n$ over $\Bbbk$,
every $\Bbbk$-point $P$ on $X$, and every finite
subgroup~$G$ of the stabilizer~\mbox{$\Aut(X;P)$} of~$P$, the order of $G$
is at most~$J_A(n)$.
\end{corollary}
\begin{proof}
Let $\Bbbk$ be an algebraically closed field of characteristic $p\ge 0$, and
let $X$ be an $n$-dimensional abelian variety over $\Bbbk$.
Set $\ell=2$ if $p\neq 2$, and $\ell=3$ if $p=2$.
According to Theorem~\ref{theorem:abelian-variety}, the group~\mbox{$\Aut(X;P)$} is isomorphic to a subgroup
of~\mbox{$\GL_{2n}(\ZZ_\ell)$}.
Therefore, the assertion follows from Lemma~\ref{lemma:Minkowski-l-adic}.
\end{proof}

\begin{remark}\label{remark:elliptic-curves-uniform}
In the notation of
Corollaries~\ref{corollary:abelian-variety-Jordan}, \ref{corollary:abelian-variety-p-Jordan},
and~\ref{corollary:torus-stabilizer},
one can replace the constant~\mbox{$J_A(1)$} by~$24$, see
\cite[Theorem~III.10.1]{Silverman}
and \cite[Proposition~A.1.2(c)]{Silverman}; actually,
in this case the whole group $\Aut(X;P)$ is finite.
Note that this bound is stronger than the bound
given by Corollaries~\ref{corollary:abelian-variety-Jordan}, \ref{corollary:abelian-variety-p-Jordan},
and~\ref{corollary:torus-stabilizer}.
\end{remark}

The assertion of Corollary~\ref{corollary:torus-stabilizer} does not hold for stabilizers of closed points
over algebraically non-closed fields, as shown
by the following example.

\begin{example}\label{example:large-stabilizer}
Let $\Bbbk$ be an algebraically closed field of characteristic~\mbox{$p\ge 0$},
let $A$ be a positive-dimensional abelian variety over $\Bbbk$ (with a chosen group structure), and let $m$ be a positive integer not divisible by~$p$.
Choose a primitive $m$-th root of unity $\zeta\in\Bbbk$, and a point $c\in A(\Bbbk)$ whose order~$m$ in the group~\mbox{$A(\Bbbk)$} equals $m$.
Consider the action of the group $\Gamma\cong\ZZ/m\ZZ$ on~\mbox{$\mathbb{A}^1\times A$}
such that the generator of $\Gamma$ acts by the transformation
$$
(t,a)\mapsto (\zeta t, a+c),
$$
where $t$ is a coordinate on $\mathbb{A}^1$, and $a\in A$.
Set~\mbox{$\mathcal{X}=(\mathbb{A}^1\times A)/\Gamma$}. Then there is a
morphism
$$
\phi\colon \mathcal{X}\to\mathbb{A}^1\cong\Spec\Bbbk[t^m]
$$
that fits into the commutative diagram
$$
\xymatrix{
\mathbb{A}^1\times A\ar@{->}[r]\ar@{->}[d] & \mathcal{X}\ar@{->}[d]^{\phi}\\
\mathbb{A}^1\ar@{->}[r] & \mathbb{A}^1
}
$$
Let $X$ be the scheme-theoretic generic fiber of $\phi$. Then $X$ is an abelian
variety over the field $\KK=\Bbbk(t^m)$.
Let $Q$ be a $\Bbbk$-point of $A$, and let~$P$ be the closed point of $X$ corresponding to the image of the section~\mbox{$\mathbb{A}^1\times\{Q\}$}
of the projection~\mbox{$\mathbb{A}^1\times A\to\mathbb{A}^1$}; thus, $P$ is not a $\KK$-point of $X$ if~\mbox{$m>1$}.
The translation by~$c$ defines a faithful action of the group~\mbox{$\ZZ/m\ZZ$} on~$X$,
and this action preserves the point~$P$. Taking $m$ arbitrarily large, we obtain a series of varieties and their closed points preserved
by automorphisms of arbitrarily large finite orders.
\end{example}

\section{Varieties of non-negative Kodaira dimension}
\label{section:varieties-kappa-non-negative}

In this section we
make some general observations on
automorphism groups of varieties of non-negative
Kodaira dimension and
prove Proposition~\ref{proposition:Hu-improved}.
The following assertions
(and the arguments that prove them) are well known to experts.
We provide their proofs to be self-contained.

\begin{theorem}[{cf. \cite[\S2.2]{Lazarsfeld}}]
\label{theorem:general-type-abundance}
Let $X$ be a smooth geometrically irreducible projective variety such that $\kappa(X)=\dim X$.
Then for some positive integer $n$ the rational map
defined by the linear system $|nK_X|$
is birational onto its image.
\end{theorem}

\begin{proof}
Since $\kappa(X)=\dim X$, there exists a positive integer $m$ such that~\mbox{$\dim \Phi_m(X)=\dim X$}, where
$\Phi_m\colon X\dashrightarrow \PP^N$ is the rational map given by the linear system~\mbox{$|mK_X|$}.
By  elimination of indeterminacy (see for instance~\mbox{\cite[Example~II.7.17.3]{Hartshorne}}),
there is a commutative diagram
$$
\xymatrix{
&  \tilde{X} \ar[dl]_{\pi} \ar[dr]^{\tilde{\Phi}_m} & \\
X \ar@{-->}[rr]^{\Phi_m}  && \PP^N
}
$$
Here $\pi$ is the blow up of the indeterminacy locus of $\Phi_m$.
Let $E\subset \tilde{X}$ be the
exceptional divisor of the morphism $\pi$.
Then $E$ is an effective Cartier divisor, and
$$
\pi^*(mK_X)\sim\tilde{\Phi}_m^*H+E,
$$
where $H$ is a hyperplane in~$\PP^N$.

Let $W=\tilde{\Phi}_m(\tilde{X})$.
Since $\tilde{\Phi}_m$ is surjective, we conclude that
$\tilde{\Phi}_m^*$ induces an injective map
$$
H^0(W,\mathcal{O}_W(H|_W))\hookrightarrow H^0(\tilde{X},\mathcal{O}_{\tilde{X}}(\tilde{\Phi}_m^*H)).
$$
Therefore, one has
\begin{multline*}
h^0(X,\mathcal{O}_X(mK_X))=h^0(\tilde{X},\mathcal{O}_{\tilde{X}}(\pi^*(mK_X)))\ge \\ \ge h^0(\tilde{X},\mathcal{O}_{\tilde{X}}(\tilde{\Phi}_m^*H))\ge  h^0(W,\mathcal{O}_W(H|_W)).
\end{multline*}

Denote $r=\dim W=\dim X=\dim \tilde{X}$.
Since~\mbox{$h^0(W,\mathcal{O}_W(m'H|_W))$} is a polynomial in $m'$ for $m'\gg 0$, there exists a positive constant $C$ such that
$$
h^0(W,\mathcal{O}_W(m'H|_W))\ge C(m')^r.
$$
Hence for $m'\gg0$ we have
\begin{equation}\label{eq:h0-growth}
h^0(X,\mathcal{O}_X(m'mK_X))\ge C(m')^r.
\end{equation}

Let $A$ be a very ample divisor on $X$. Consider the exact sequence
$$
0\to\mathcal{O}_X(-A)\to\mathcal{O}_X\to\mathcal{O}_A\to 0.
$$
It gives the exact sequence
$$
0\to\mathcal{O}_X(m'mK_X-A)\to\mathcal{O}_X(m'mK_X)\to\mathcal{O}_A(m'mK_X|_A)\to 0.
$$
Since $\dim A=r-1$, we know that $h^0(X,\mathcal{O}_A(m'mK_X|_A))$ grows as a polynomial of degree
at most $r-1$ in $m'$. Therefore, from~\eqref{eq:h0-growth}
we get
$$
h^0(X, \mathcal{O}_X(m'mK_X-A))\neq 0
$$
for $m'\gg 0$.

Let $F\in |m'mK_X-A|$
be an effective divisor, so that~\mbox{$m'mK_X\sim A+F$.} Over the open subset $U=X\setminus F\subset X$,
the linear system $|A|$ can be viewed as a linear subsystem of $|m'mK_X|$
(where the embedding is given by $L\mapsto L+F$ for~\mbox{$L\in |A|$}). By assumption, $A$ is very ample, so the rational map $\Phi_{m'm}$ induced by $|m'mK_X|$ is an embedding on $U$. Therefore, the map  $\Phi_{m'm}$ is birational.
\end{proof}

\begin{lemma}\label{lemma:Gm-or-Ga}
Let $\Gamma$ be a non-trivial connected linear algebraic group
over an algebraically closed field.
Then $\Gamma$ contains a subgroup isomorphic either
to~$\mathbb{G}_{\mathrm{m}}$ or to~$\mathbb{G}_{\mathrm{a}}$.
\end{lemma}

\begin{proof}
Let $R$ be the radical of $\Gamma$, i.e. a maximal closed, connected, normal,
solvable subgroup of~$\Gamma$. First, suppose that $R$ is non-trivial.
If $R$ is not a torus, then it contains a subgroup
isomorphic to~$\mathbb{G}_{\mathrm{a}}$
by~\mbox{\cite[Lemma~6.3.4]{Springer}}.
If $R$ is a torus, then it contains a subgroup
isomorphic to~$\mathbb{G}_{\mathrm{m}}$.
Thus, we may assume that $R$ is trivial, so that
$\Gamma$ is semi-simple. In this case $\Gamma$ is generated
by its maximal torus~$T$ and a certain collection~$\{U_\alpha\}$ of subgroups
isomrophic to~$\mathbb{G}_{\mathrm{a}}$,
see~\mbox{\cite[Proposition~8.1.1]{Springer}}.
Since $\Gamma$ is non-trivial, we conclude that
either the torus~$T$ is positive-dimensional, or the collection~$\{U_\alpha\}$
is non-empty. Thus, one finds a subgroup of $\Gamma$
isomorphic to~$\mathbb{G}_{\mathrm{m}}$ or
to~$\mathbb{G}_{\mathrm{a}}$, respectively.
\end{proof}

\begin{proposition}[{cf. \cite[Theorem~14.1]{Ueno}, \cite[Proposition~7.1.4]{Brion-long}}]
\label{proposition:Aut-nonruled}
Let~$X$ be a
smooth geometrically irreducible projective variety of non-negative Kodaira dimension
over a field~$\Bbbk$. Then the group scheme~$\AAut_X$ does not contain
non-trivial connected
linear algebraic subgroups.
\end{proposition}

\begin{proof}
Since a connected algebraic group is geometrically connected, we may assume that the field $\Bbbk$
is algebraically closed.
Suppose that $\AAut_X$ contains a non-trivial connected linear algebraic group.
Then it contains a linear algebraic group $\mathbb{G}$ isomorphic either
to $\mathbb{G}_{\mathrm{m}}$ or to $\mathbb{G}_{\mathrm{a}}$,
see Lemma~\ref{lemma:Gm-or-Ga}.
Hence~$X$ is birational to~\mbox{$Y\times \PP^1$}, where $Y$ is the
geometric quotient of some open subset of $X$ by the action of~$\mathbb{G}$;
see \cite{Popov-quotient} or \cite[Lemma~A.5]{SV18}.
By Lemma~\ref{lemma:ruled-kappa} one has
$$
\kappa(X)=-\infty,
$$
which contradicts our assumptions.
\end{proof}

The following corollary of
Theorem~\ref{theorem:general-type-abundance}
and Proposition~\ref{proposition:Aut-nonruled} will be used only in
dimensions $1$ and $2$ in this paper. However, we
state it for arbitrary dimension because its proof
is essentially the same.

\begin{corollary}[{see \cite{Matsumura}, \cite[Corollaire~4]{MDLM},
cf. \cite[Corollary~14.3]{Ueno}}]
\label{corollary:general-type}
Let $X$ be a smooth geometrically irreducible
projective variety over a
field $\Bbbk$ such that $\kappa(X)=\dim X$.
Then $\AAut_X$ is a finite group scheme,
and~\mbox{$\Aut(X)$} is a finite group.
\end{corollary}

\begin{proof}
We may assume that the field $\Bbbk$ is algebraically closed.
For some $m>0$, the rational map
$$
\phi_m\colon X\dasharrow Z\subset \mathbb{P}^n
$$
defined by the linear system $|mK_X|$
is birational, see Theorem~\ref{theorem:general-type-abundance}.
This realizes the group scheme~\mbox{$\AAut_X$} as the stabilizer
of $Z$ in the linear algebraic group~\mbox{$\PGL_{n+1}$}.
Therefore, $\AAut_X$ is a group scheme of finite type, and
the reduced neutral component $\AAut^0_{X,\red}$ of $\AAut_X$ is a linear
algebraic group. Moreover, $\AAut^0_{X,\red}$ is trivial
by Proposition~\ref{proposition:Aut-nonruled}. This implies
that $\AAut_X$ is a finite group scheme, and in particular
$\Aut(X)$ is a finite group.
\end{proof}

We will need the following general fact concerning algebraic groups.

\begin{theorem}[{see e.g. \cite[Lemma~4.3.1]{Brion-long}}]
\label{theorem:alg-group-abelian}
Let $\Gamma$ be a connected algebraic group.
Suppose that $\Gamma$ does not
contain non-trivial connected linear algebraic subgroups.
Then $\Gamma$ is an abelian variety.
\end{theorem}

\begin{corollary}\label{corollary:Aut-0-abelian}
Let $X$ be a smooth geometrically irreducible
projective variety over a perfect field.
Suppose that $\kappa(X)\ge 0$. Then $\AAut^0_{X,\red}$ is an abelian variety.
\end{corollary}

\begin{proof}
The connected algebraic group $\AAut^0_{X,\red}$
does not contain non-trivial connected linear algebraic subgroups
by Proposition~\ref{proposition:Aut-nonruled}.
Therefore, the required assertion follows from
Theorem~\ref{theorem:alg-group-abelian}.
\end{proof}

Now we are ready to prove Proposition~\ref{proposition:Hu-improved}.

\begin{proof}[Proof of Proposition~\ref{proposition:Hu-improved}]
We may assume that $X$ is defined over an algebraically closed field.
Then the group~\mbox{$\Aut^0(X)$} is the group of points of
some abelian variety by Corollary~\ref{corollary:Aut-0-abelian}.
In particular, it is abelian,
On the other hand, the quotient group
$\Aut(X)/\Aut^0(X)$ has bounded finite subgroups by Lemma~\ref{lemma:MZ}.
Hence the group $\Aut(X)$ is Jordan
by Lemma~\ref{lemma:group-theory}.
\end{proof}

\section{Jordan property for curves}
\label{section:curves}

In this section we describe the Jordan properties for
automorphism groups of smooth projective curves.

\begin{lemma}\label{lemma:curves}
Let $\Bbbk$ be a field of characteristic $p>0$,
and let $C$ be a smooth geometrically
irreducible projective curve of genus $g$ over $\Bbbk$.
The following assertions hold.
\begin{itemize}
\item[(i)]
If $g=0$, then the group $\Aut(C)$ is $p$-Jordan.

\item[(ii)] If $g=1$, then the group $\Aut(C)$ is Jordan.

\item[(iii)] If the $g\ge 2$, then the group $\Aut(C)$ is finite.
\end{itemize}
\end{lemma}

\begin{proof}
We may assume that $\Bbbk$ is algebraically closed. Now
assertion~(i) follows from Theorem~\ref{theorem:Hu}.
Assertion~(ii) is given by Corollary~\ref{corollary:abelian-variety-Jordan}.
Assertion~(iii) is given by Corollary~\ref{corollary:general-type}.
\end{proof}

Similarly to the Hurwitz bound over fields of characteristic zero,
there exists a bound
on the order of the automorphism
group of a curve of genus $g\ge 2$ over a field of positive characteristic
that depends only on~$g$,
but not on the characteristic of the field,
see~\cite{St73}. In the case of elliptic curves,
the constant arising in Lemma~\ref{lemma:curves}(ii) is always bounded
by~$24$, see Remark~\ref{remark:elliptic-curves-uniform}. For a more explicit
version of Lemma~\ref{lemma:curves}(i), one can use a classification of finite groups acting on~$\PP^1$.

\begin{theorem}[{see e.g. \cite[Theorem~2.1]{DD}}]
\label{theorem:ADE}
Let $\Bbbk$ be a field of characteristic~\mbox{$p>0$}, let
$G\subset\PGL_2(\Bbbk)$
be a finite group, and let $G_p$ be a $p$-Sylow subgroup of~$G$.
Then $G$ is one of the following groups:
\begin{itemize}
\item[(1)] a dihedral group of order $2n$, where $n>1$ is coprime to~$p$;

\item[(2)] one of the groups $\mathfrak{A}_4$, $\mathfrak{S}_4$, or
$\mathfrak{A}_5$;

\item[(3)] the group $\PSL_2(\mathbf{F}_{p^k})$ for some $k\ge 1$;

\item[(4)] the group $\PGL_2(\mathbf{F}_{p^k})$ for some $k\ge 1$;

\item[(5)] a group of the form $G_p\rtimes \ZZ/n\ZZ$, where $n\ge 1$
is coprime to~$p$, and $G_p$ is a $p$-subgroup of the additive group of~$\Bbbk$.
\end{itemize}
\end{theorem}

%In particular, Theorem~\ref{theorem:ADE} tells us that $\PGL_2(\Bbbk)$ contains
%all cyclic groups~\mbox{$\ZZ/n\ZZ$} with $n$ coprime to $p$,
%which can be regarded as groups of type~(5) with trivial $p$-Sylow
%subgroup $G_p$. Also, $\PGL_2(\Bbbk)$ always contains the group~\mbox{$\ZZ/2\ZZ$},
%which can be regarded as a group of type~(5)
%both if $p=2$ and $p\neq 2$.

\begin{remark}\label{remark:PGL2-centralizer}
Let $\Bbbk$ be a field of characteristic $p>0$,
and let~\mbox{$g\in\PGL_2(\Bbbk)$} be a non-trivial element of finite order $n$ coprime to $p$.
Then $g$ is semi-simple, and thus is contained in some algebraic torus $T\subset\PGL_2(\Bbbk)$.
Moreover, if $n\ge 3$, then the centralizer of $g$ coincides with
$T$. In particular, every finite group that commutes with $g$ is
cyclic. Moreover, the order of such a group is coprime to~$p$, because an algebraic torus over $\Bbbk$
does not contain elements of order~$p$. This provides additional restrictions
for the possible structure of the groups of type~(5) in Theorem~\ref{theorem:ADE}.
\end{remark}

\begin{remark}\label{remark:PGL-generators}
Any group of one of types (1)--(4) in the notation of Theorem~\ref{theorem:ADE}
can be generated by $2$ elements.
Indeed, this clearly holds for a dihedral group. Also, it is well known that the symmetric group
on $n$ elements is generated by a cycle of length~$2$ and a cycle of length~$n$.
Furthermore, any non-abelian finite simple group is generated by two elements, see e.g.~\cite{King}.
In particular, this holds for an alternating group $\mathfrak{A}_5$. A similar assertion for
the group $\mathfrak{A}_4$ can be checked directly.
The groups $\PSL_2(\mathbf{F}_{p^k})$ are simple except for~\mbox{$\PSL_2(\mathbf{F}_2)\cong\mathfrak{S}_3$} and
$\PSL_2(\mathbf{F}_3)\cong\mathfrak{A}_4$, see \cite[\S3.3.1]{Wilson}.
Thus, any group of this type can be generated by two elements.
Finally, for the group~\mbox{$\PGL_2(\mathbf{F}_{p^k})$} the required assertion follows from~\cite{Waterhouse}.
\end{remark}

Now we can give an explicit bound for the index of a normal abelian subgroup in
a finite group acting on~$\PP^1$.

\begin{lemma}\label{lemma:P1}
For every prime $p$ and every field $\Bbbk$ of characteristic $p$,
every finite subgroup
$$
G\subset \Aut(\PP^1)\cong \PGL_2(\Bbbk)
$$
contains a characteristic cyclic subgroup
of order coprime
to $p$ and index at most~\mbox{$60|G_{p}|^{3}$},
where $G_{p}$ is a $p$-Sylow subgroup of $G$.
\end{lemma}

\begin{proof}
Let $\Bbbk$ be a field of characteristic $p>0$, let
$G\subset\PGL_2(\Bbbk)$
be a finite group, and let $G_p$ be a $p$-Sylow subgroup of $G$.
Then $G$ is a group of one of
types (1)--(5) in the notation of Theorem~\ref{theorem:ADE}.

If $G$ is a dihedral group of type~(1),
then the commutator subgroup of $G$ is a characteristic cyclic subgroup of
index at most $4$ in $G$, and its order is coprime to $p$.
If $G$ is one of the groups $\mathfrak{A}_4$, $\mathfrak{S}_4$, or
$\mathfrak{A}_5$, then the trivial subgroup has index at most $60$ in~$G$.
If $G\cong\PSL_2(\mathbf{F}_{p^k})$ and $p\neq 2$, then
the trivial subgroup of $G$ has index
$$
|G|=\frac{|\SL_2(\mathbf{F}_{p^k})|}{2}=\frac{p^k(p^{2k}-1)}{2}
<p^{3k}=|G_p|^3.
$$
If $G\cong\PSL_2(\mathbf{F}_{2^k})$, then
the trivial subgroup of $G$ has index
$$
|G|=|\SL_2(\mathbf{F}_{2^k})|=2^k(2^{2k}-1)<2^{3k}=|G_2|^3.
$$
If $G\cong\PGL_2(\mathbf{F}_{p^k})$, then the trivial subgroup
of $G$ has index
$$
|G|=p^k(p^{2k}-1)<p^{3k}=|G_p|^3.
$$

Now suppose that $G$ is of type~(5).
Then $G$ contains a cyclic group~\mbox{$L\cong\ZZ/n\ZZ$}
such that $G\cong G_p\rtimes L$,
and $G_p\cong (\ZZ/p\ZZ)^m$ for some positive integer~$m$.
Let~$L'$ be the centralizer of $G_p$ in $L$.
By Lemma~\ref{lemma:Darafsheh}, the index of $L'$ in $L$
is at most~\mbox{$p^m-1$}. Thus,
$L'$ is a cyclic subgroup of $G$ of index at most
$$
|G_p|(p^m-1)<|G_p|^2,
$$
and the order of $L'$ is coprime to $p$.

We claim that $L'$ is a characteristic subgroup
of $G$. Indeed, $G_p$ is characteristic in $G$ by
Example~\ref{example:characteristic-product}.
Thus its centralizer~\mbox{$C(G_p)$}
in $G$ is also a characteristic subgroup of $G$. On the other hand, one has
$$
C(G_p)\cong G_p\times L'.
$$
Therefore, $L'$ is characteristic in $C(G_p)$ (cf. Example~\ref{example:characteristic-product}), and hence also
characteristic in~$G$.
\end{proof}

\section{Cremona group}
\label{section:Cremona}

In this section we prove Theorem~\ref{theorem:Serre}.
Our proof is (nearly) identical to that presented in
\cite{Serre-2009}.

We start by recalling the assertion proved
in~\mbox{\cite[Lemma~5.2]{Serre-2009}}.
We provide its detailed proof for the convenience of the reader.

\begin{lemma}\label{lemma:Serre}
Let $n$ be a positive integer, and let $\Bbbk$ be a field of characteristic~$p$
such that $p$ does not divide $n$ and $\Bbbk$ contains a primitive $n$-th root of~$1$.
Let $S$ be a smooth geometrically irreducible projective surface over $\Bbbk$, and let~\mbox{$H\subset\Aut(S)$} be a finite group.
Suppose that there exists an $H$-equivariant conic bundle~\mbox{$\phi\colon S\to C$}.
Denote by $F$ the subgroup of~$H$ that consists of automorphisms that are fiberwise with respect to~$\phi$
(see Definition~\ref{definition:fiberwise}).
Let~\mbox{$R\subset F$} be a cyclic group of order~$n$, and let $\alpha\in H$ be an element normalizing~$R$. Then~$\alpha^2$ commutes with~$R$.
\end{lemma}

\begin{proof}
Let $S_\eta$ be the fiber of $\phi$ over the scheme-theoretic generic
point $\eta$ of~$C$.
Then $S_\eta$ is a smooth conic over the field~\mbox{$\KK=\Bbbk(C)$}.
The group $F$ can be considered as a subgroup of $\Aut(S_\eta)\subset\PGL_2(\bar{\KK})$.
Thus, the cyclic group $R$ has exactly two fixed $\bar{\KK}$-points on the conic $S_{\eta, \bar{\KK}}\cong \PP^1_{\bar{\KK}}$.
Denote them by~$P_+$ and~$P_-$. Recall that the action of $R$
in the Zariski tangent spaces
$T_{P_\pm}(S_{\eta, \bar{\KK}})$ is faithful by Theorem~\ref{theorem:fixed-point}.
Thus, there is a primitive $n$-th root $\zeta$ of~$1$
such that a generator~$g$ of $R$ acts on $T_{P_\pm}(S_{\eta, \bar{\KK}})\cong\bar{\KK}$ by $\zeta^{\pm 1}$.
For any positive integer $r$, the element~$g^r$ acts on $T_{P_\pm}(S_{\eta, \bar{\KK}})$ by $\zeta^{\pm r}$;
thus, an element of $R$ is uniquely defined by its action in any of these two Zariski tangent spaces.

One can consider $\alpha$ as an automorphism of the scheme $S_{\eta, \bar{\KK}}$ (of non-finite type) over the field $\Bbbk$.
Write $\alpha g\alpha^{-1}=g^r$. Since the fixed points of $g^r$
on~$S_{\eta, \bar{\KK}}$ are $P_+$ and $P_-$, we see that $\alpha(P_+)\in\{P_+,P_-\}$.
Hence $\alpha^2(P_+)=P_+$, and
$$
g^{r^2}=\alpha^2g\alpha^{-2}
$$
acts on~$T_{P_+}(S_{\eta, \bar{\KK}})$ by $\zeta^{r^2}=\bar{\alpha}^2(\zeta)$, where $\bar{\alpha}$ is the automorphism
of the field $\KK$ over $\Bbbk$ induced by $\alpha$. However, $\zeta$ is contained in $\Bbbk$, so
$\bar{\alpha}(\zeta)=\zeta$. Therefore, we see that $r^2$ is congruent to $1$ modulo $n$, which means that $\alpha^2$ commutes with~$g$.
\end{proof}

The next lemma allows one to deal with finite groups
acting on surfaces with conic bundle structure.
It also provides an additional observation concerning the number of generators of abelian
subgroups in such groups.

\begin{lemma}\label{lemma:conic-bundle-Jordan}
Let $\Bbbk$ be a field of characteristic $p>0$, and let
$S$ be a smooth geometrically rational projective surface over~$\Bbbk$.
Let $G\subset\Aut(S)$ be a finite group, and let
$\phi\colon S\to C$ be a $G$-equivariant conic bundle.
Then $G$ contains a normal abelian subgroup generated by at most two elements that has order coprime
to $p$ and index at most~\mbox{$7200|G_{p}|^{3}$},
where~$G_{p}$ is a $p$-Sylow subgroup of~$G$.
\end{lemma}

\begin{proof}
We can assume that the field $\Bbbk$ is algebraically closed;
in particular, one has $C\cong\PP^1$.

There is an exact sequence of groups
$$
1\to F\to G\to \bar{G}\to 1,
$$
where the action of $F$ is fiberwise
with respect to $\phi$, while $\bar{G}$ acts faithfully on~$\PP^1$.
By Lemma~\ref{lemma:P1} there exists a normal
cyclic subgroup $\bar{H}$ in $\bar{G}$
of order coprime to~$p$ whose index in $\bar{G}$ does not exceed
$60|\bar{G}_p|^3$, where $\bar{G}_p$ is a $p$-Sylow subgroup of $\bar{G}$.
Let $H$ be the preimage of $\bar{H}$ in $G$, so that there
is an exact sequence of groups
$$
1\to F\to H\to \bar{H}\to 1.
$$
In particular, $H$ is a normal subgroup of $G$.

By Lemma~\ref{lemma:P1} there exists a characteristic cyclic
subgroup $R$ of order coprime to $p$
and index at most $60|F_p|^3$ in $F$, where $F_p$
is a $p$-Sylow subgroup of $F$. The group $H$ acts on $F$ by conjugation,
and this action preserves the characteristic subgroup $R$.

Pick an element $\alpha$ of $H$
such that its image $\bar{\alpha}$ in $\bar{H}$ generates $\bar{H}$.
Since~$|\bar{H}|$ is coprime to $p$, the element $\bar{\alpha}^p$ generates $\bar{H}$ as well.
Hence, replacing $\alpha$ by its appropriate power, we may assume that
the order of $\alpha$ is coprime to~$p$.
Since~$\alpha$ normalizes the subgroup $R$, and the (algebraically closed) field
$\Bbbk$ contains a primitive root of $1$ or degree $|R|$, by Lemma~\ref{lemma:Serre}
the element $\alpha^2$ commutes with~$R$.
Let $\tilde{A}$ be the subgroup of $H$ generated by $R$ and $\alpha^2$. Then
$\tilde{A}$ is abelian, and its order is coprime to $p$.
Note that the subgroup $\bar{H}_{\bar{\alpha}^2}\subset\bar{H}$ generated by~$\bar{\alpha}^2$
either coincides with $\bar{H}$, or has index $2$ in $\bar{H}$.

The subgroup $R$ is characteristic in $F$, and hence
normal in $G$.
Also, the subgroup $\bar{H}_{\bar{\alpha}^2}$ is characteristic in $\bar{H}$, because
the cyclic group $\bar{H}$ contains at most one subgroup of given order
(cf. Example~\ref{example:characteristic-by-generators}).
Hence $\bar{H}_{\bar{\alpha}^2}$ is normal in~$\bar{G}$.
Still we cannot conclude from this that
$\tilde{A}$ is normal in~$G$.
However, let $A$ be the intersection of all subgroups
of $G$ conjugate to $\tilde{A}$. Then $A$ is an abelian group of order coprime to $p$, and it is normal in $G$.
Since $\tilde{A}$ is generated by two elements, the same holds for $A$.
Also, we know that $A$ contains the subgroup~$R$, because~$R$ is normal in $G$.
Similarly, the image of $A$ in $\bar{H}$ coincides with $\bar{H}_{\bar{\alpha}^2}$,
because~$\bar{H}_{\bar{\alpha}^2}$ is normal in $\bar{G}$.
Therefore, the index of
$A$ in $G$ is
$$
\frac{|G|}{|A|}=
\frac{|\bar{G}|}{|\bar{H}_{\bar{\alpha}^2}|}\cdot \frac{|F|}{|A\cap F|}\le
2\frac{|\bar{G}|}{|\bar{H}|}\cdot \frac{|F|}{|R|}\le
7200|\bar{G}_p|^3\cdot |F_p|^3=7200|G_p|^3.
$$
\end{proof}

Finally, we are ready to prove our main result.

\begin{proof}[Proof of Theorem~\ref{theorem:Serre}]
Let $\Bbbk$ be a field of characteristic $p>0$.
We may assume that $\Bbbk$ is algebraically closed.
Let $G$ be a finite subgroup of $\Bir(\PP^2)$.
By Theorem~\ref{theorem:MMP}
there exists
a surface $S$ with a faithful regular action of $G$,
such that~$S$ is either a del Pezzo surface
or a $G$-equivariant conic bundle.

Suppose that $S$ is a del Pezzo surface.
By Theorem~\ref{theorem:del-Pezzo} the linear system~\mbox{$|-3K_S|$}
defines an embedding $\iota\colon S\hookrightarrow\PP^N$,
where $N\le 54$. Since the linear system $|-3K_S|$ is $G$-invariant, we see
that the map $\iota$ is $G$-equivariant, and thus~$G$
can be realized as a subgroup of $\PGL_{55}(\Bbbk)$.
By Corollary~\ref{corollary:PGL}
there exists a constant $J_{dP}$ independent of~$p$, $\Bbbk$, and~$G$,
such that $G$ contains a normal abelian subgroup of order coprime
to $p$ and index at most $J_{dP}\cdot |G_{p}|^{3}$,
where $G_{p}$ is a $p$-Sylow subgroup of $G$.

Now we are left with the case when $S$ is a
$G$-equivariant conic bundle, which is covered by Lemma~\ref{lemma:conic-bundle-Jordan}.
\end{proof}

\section{Jordan property for surfaces}
\label{section:Popov}

In this section we prove Theorem~\ref{theorem:Popov}.
The proof of the following assertion is identical to
that in \cite{Zarhin-2014}.

\begin{lemma}\label{lemma:Zarhin}
Let $\Bbbk$ be an algebraically closed field of characteristic $p>0$.
Let~$A$ be an abelian variety of positive dimension~$n$ over~$\Bbbk$,
and let~\mbox{$S=A\times\PP^1$}. Then the group
$\Bir(S)$ is not generalized $p$-Jordan.
\end{lemma}

\begin{proof}
Choose a prime number $\ell\neq p$. Fix a $\Bbbk$-point $O\in A$ to define a group structure on $A$.
For every $\Bbbk$-point $x\in A$, let $t_x\colon A\to A$ be the translation
by~$x$.

Let $\LLL_0$ be an ample line bundle on $A$,
and set $\LLL=\LLL_0^{\otimes \ell}$. Let $K(\LLL)$ denote the group
of $\Bbbk$-points $x\in A$ such that $t_x^*\LLL\cong\LLL$, and let
$K(\LLL)_\ell$ be its subgroup that consists of all elements whose order is a power of $\ell$.
According to~\mbox{\cite[\S23, Theorem~3]{Mumford-AbelianVarieties}},
the group $K(\LLL)$ contains all the $\ell$-torsion $\Bbbk$-points of $A$.
Since $\ell\neq p$, the group
of $\ell$-torsion $\Bbbk$-points of $A$ is isomorphic
to~\mbox{$(\ZZ/\ell\ZZ)^{2n}$},
see e.g.~\mbox{\cite[\S6]{Mumford-AbelianVarieties}}.
Thus, the group $K(\LLL)_\ell$ is non-trivial.
Furthermore, since $\LLL$ is ample, the groups
$K(\LLL)$ and $K(\LLL)_\ell$ are finite by~\mbox{\cite[\S23, Theorem~4]{Mumford-AbelianVarieties}}.

Let $\Tot(\LLL)$ be the total space of $\LLL$.
For every $x\in K(\LLL)$, there exists a (non-unique) fiberwise linear
isomorphism~\mbox{$\tilde{t}_x^*\colon \Tot(\LLL)\to \Tot(\LLL)$} that fits into a commutative diagram
$$
\xymatrix{
\Tot(\LLL)\ar@{->}[d]\ar@{->}[rr]^{\tilde{t}_x^*} && \Tot(\LLL)\ar@{->}[d]\\
A\ar@{->}[rr]^{t_x} && A
}
$$
Moreover, there exists an exact sequence of groups
$$
1\to\Bbbk^*\to \GGG(\LLL)\to K(\LLL)\to 1,
$$
where $\GGG(\LLL)$ is a central extension of $K(\LLL)$ acting by automorphisms of $\Tot(\LLL)$ of the form~$\tilde{t}_x^*$.
The preimage of $K(\LLL)_\ell$ in $\GGG(\LLL)$ contains a finite group
$\HHH_{\ell}$ that is a central extension of
$K(\LLL)_\ell$ by a cyclic group $\mathcal{Z}$ of order~$\ell^k$ for some non-negative integer~$k$.

For every two
elements $x, y\in K(\LLL)_\ell$, let $(x,y)$ be their commutator pairing; in other words, write
$$
\tilde{x}\tilde{y}\tilde{x}^{-1}\tilde{y}^{-1}=(x,y)z,
$$
where $\tilde{x}$ and $\tilde{y}$ are (arbitrary) preimages
of $x$ and $y$ in $\HHH_\ell$, and
$z$ is a generator of the cyclic group~\mbox{$\mathcal{Z}\subset\HHH_{\ell}$}.
According to~\mbox{\cite[\S23, Theorem~4]{Mumford-AbelianVarieties}},
the commutator pairing is non-trivial on $K(\LLL)_\ell$.
Hence $\HHH_{\ell}$ is a non-abelian group whose order is a power of the prime number~$\ell$.
Thus every abelian subgroup of~$\HHH_{\ell}$ has index at least $\ell$
in~$\HHH_{\ell}$.

It remains to notice that the $\mathbb{A}^1$-bundle
$\Tot(\LLL)\to A$ is birational to~\mbox{$S=A\times\PP^1$}.
Therefore, the group $\Bir(S)$ contains a group~$\HHH_{\ell}$ for every prime
$\ell\neq p$, and the required assertion follows.
\end{proof}

\begin{lemma}\label{lemma:ruled}
Let $\Bbbk$ be an algebraically closed field of characteristic $p>0$.
Let~$C$ be a smooth geometrically irreducible projective curve of genus
at least $2$ over $\Bbbk$, and set $S=C\times\PP^1$.
Then the group~\mbox{$\Bir(S)$} is $p$-Jordan.
\end{lemma}

\begin{proof}
According to Lemma~\ref{lemma:ruled-Bir}, the group
$\Bir(S)$ fits into the exact sequence
\begin{equation*}
1\to \Bir(S)_\phi\to \Bir(S)\to\Gamma,
\end{equation*}
where $\Bir(S)_\phi\subset\PGL_2(\Bbbk(C))$
and $\Gamma\subset\Aut(C)$.
Since the genus of $C$ is at least $2$, the group~$\Gamma$ is finite by Lemma~\ref{lemma:curves}(iii).
On the other hand, the group~\mbox{$\PGL_2(\Bbbk(C))$}, and thus also the group~\mbox{$\Bir(S)_\phi$}, is $p$-Jordan by
Corollary~\ref{corollary:PGL}.
Therefore, the required assertion follows from Lemma~\ref{lemma:group-theory}.
\end{proof}

Now we can prove the main result of this section.

\begin{proof}[Proof of Theorem~\ref{theorem:Popov}]
Assertion (i) is given by Lemma~\ref{lemma:Zarhin}.
In the rest of the proof we may assume that $S$ is smooth and projective
by Corollary~\ref{corollary:smooth-projective-model}.
If $\kappa(S)=-\infty$, then $S$ is birational to a product $C\times\PP^1$ for some irreducible smooth projective curve $C$,
see Theorem~\ref{theorem:KE}(i). Therefore, if $C$ is not an elliptic curve, then it is either rational, or has genus at least $2$,
and so the group $\Bir(S)$ is $p$-Jordan by Theorem~\ref{theorem:Serre} and Lemma~\ref{lemma:ruled}.
On the other hand, in each of these cases $\Bir(S)$ is not Jordan because it contains a subgroup isomorphic to~\mbox{$\Aut(\PP^1)\cong\PGL_2(\Bbbk)$}
which is not Jordan. This proves assertion~(ii).
Thus, we may assume that the Kodaira dimension of $S$
is non-negative, and replace~$S$ by its minimal model.
Then~\mbox{$\Bir(S)=\Aut(S)$} by Lemma~\ref{lemma:Bir-vs-Aut},
so that assertion~(iii) follows from
Proposition~\ref{proposition:Hu-improved}.
\end{proof}

\begin{remark}
There are alternative ways that
do not use Proposition~\ref{proposition:Hu-improved}
to prove Theorem~\ref{theorem:Popov}(iii) for many birational classes
of surfaces appearing in
the Kodaira--Enriques classification. For $\kappa(S)=0$ the assertion
follows from Proposition~\ref{proposition:Jordan-kappa-0}
(which we will prove later in Section~\ref{section:kappa-0} independently
of Theorem~\ref{theorem:Popov}).
For elliptic surfaces with $\kappa(S)=1$, the assertion is given by~\mbox{\cite[Corollary~1.6]{Gu}}.
It was communicated to us by Yi Gu that a result similar to~\mbox{\cite[Corollary~1.6]{Gu}}
can be established also for quasi-elliptic surfaces with~\mbox{$\kappa(S)=1$}.
Finally, for $\kappa=2$ the assertion
follows from Corollary~\ref{corollary:general-type}.
\end{remark}

\section{Surfaces of zero Kodaira dimension}
\label{section:kappa-0}

In this section we study automorphism groups of surfaces of zero Kodaira dimension and
prove
Propositions~\ref{proposition:Jordan-kappa-0} and~\ref{proposition:stabilizer-kappa-0}.
We start with the case of $K3$ and Enriques surfaces.

\begin{lemma}
\label{lemma:K3}
There exists a constant $B_{K3}$ such that for every algebraically closed
field $\Bbbk$,
every surface $S$ over $\Bbbk$
such that~$S$ is either a $K3$ surface or an Enriques surface, and for every finite subgroup~\mbox{$G\subset \Aut(S)$}, the order of~$G$
is at most~$B_{K3}$.
\end{lemma}

\begin{proof}
Let $S$ be either a $K3$ surface or an Enriques surface
over a field $\Bbbk$ of characteristic $p$.
By \cite[Corollary~2.8]{ProkhorovShramov-2019} we may assume
that $p>0$. Set~\mbox{$\ell=2$} if $p\neq 2$, and $\ell=3$ if $p=2$.
Set
$$
H(S)= H^2_{\acute{e}t}(S,\mathbb{Z}_{\ell})/T,
$$
where $T$ is the torison subgroup of $H^2_{\acute{e}t}(S,\mathbb{Z}_{\ell})$;
note that $T$ is trivial in the~$K3$ case.
Then $H(S)\cong \mathbb{Z}_\ell^b$, where $b=22$ if
$S$ is a $K3$ surface (see \cite[\S7.2]{Liedtke}), and $b=10$
if $S$ is an Enriques surface (see \cite[\S7.3]{Liedtke}).

Consider the representation
\[
\rho\colon \Aut(S)\to \GL\big(\mathrm{H}(S)\big).
\]
By Lemma~\ref{lemma:Minkowski-l-adic},
the order of every finite subgroup in the image of $\rho$ is
bounded by the constant $J_A(11)$, see~\eqref{eq:JA}.
On the other hand, the kernel of the homomorphism~$\rho$
is trivial in the case when $S$ is a $K3$
surface, see \cite[Theorem~1.4]{Keum} (cf.~\cite[Corollary~2.5]{Ogus});
and this kernel has order at most~$4$
if $S$ is an Enriques surface, see \cite[Theorem]{DM}.
\end{proof}

Next, we deal with hyperelliptic and quasi-hyperelliptic surfaces.
Recall that a minimal surface $S$
over an algebraically closed field $\Bbbk$ (of arbitrary characteristic)
is called \emph{hyperelliptic}, if $\kappa(S)=0$,
and the fibers of the Albanese map of $S$ are smooth elliptic curves.
Similarly, $S$ is called \emph{quasi-hyperelliptic}, if~\mbox{$\kappa(S)=0$},
and the fibers of the Albanese map of $S$ are singular rational curves.
The latter case is possible only if the characteristic of $\Bbbk$
equals $2$ or~$3$.

\begin{lemma}\label{lemma:hyperelliptic-groups}
There exists a constant $B_{hyp}$ with the following property.
Let~$\Bbbk$ be an algebraically closed
field, let $S$ be a hyperelliptic or a quasi-hyperelliptic surface
over $\Bbbk$, and
let~$G$ be a (possibly infinite)
subgroup of $\Aut(S)$. Then the following assertions hold.
\begin{itemize}
\item[(i)] The group $G$ contains a normal abelian subgroup of index
at most $B_{hyp}$.

\item[(ii)] If $G$ fixes a $\Bbbk$-point on $S$, then $|G|\le B_{hyp}$.
\end{itemize}
\end{lemma}
\begin{proof}
Let $S$ be a hyperelliptic or a quasi-hyperelliptic surface over
an algebraically closed field $\Bbbk$.
The Albanese morphism $\phi\colon S\to C$
maps $S$ to an elliptic curve $C$,
see \cite[Theorem~8.6]{Badescu}.
Furthermore, there exists
an elliptic curve $E$, a curve $F$ that is either an elliptic curve or a cuspidal rational
curve, and a finite group scheme $\Gamma$ acting faithfully on $E$ by translations and acting faithfully on $F$,
such that $S$ is included in the commutative diagram
$$
\xymatrix{
E\times F\ar@{->}[r]\ar@{->}[d]  & S\ar@{->}[d]^{\phi}\\
E\ar@{->}[r] & C
}
$$
We refer the reader to \cite[\S3]{BM-II} and \cite[\S2]{BM-III} for details.
In particular, this construction implies that the group~\mbox{$C(\Bbbk)$} of $\Bbbk$-points of the elliptic curve~\mbox{$C\cong E/\Gamma$}
is a normal subgroup of~\mbox{$\Aut(S)$}, and its action on $S$ agrees with the action of~$C(\Bbbk)$ on~$C$ by translations.

By the classification of automorphism groups of
hyperelliptic and quasi-hyperelliptic surfaces provided in~\cite{Martin}, there is a constant~$B_{hyp}$ independent of~$\Bbbk$ and~$S$ such that the index
of $C(\Bbbk)$ in $\Aut(S)$ does not exceed~$B_{hyp}$.
This means that every subgroup $G$ of $\Aut(S)$ has a normal abelian
subgroup of index at most~$B_{hyp}$, which gives assertion~(i).

Now suppose that a subgroup $G\subset\Aut(S)$ fixes
a $\Bbbk$-point on $S$. Since $\phi$ is equivariant with respect to $\Aut(S)$,
we conclude that $G$ acts on $C$ with a fixed $\Bbbk$-point.
Therefore, the intersection of $G$ with $C(\Bbbk)$ is trivial.
Hence~\mbox{$|G|\le B_{hyp}$}, which proves assertion~(ii).
\end{proof}

Now we prove Propositions~\ref{proposition:Jordan-kappa-0} and~\ref{proposition:stabilizer-kappa-0}.

\begin{proof}[Proof of Proposition~\ref{proposition:Jordan-kappa-0}]
Let $S$ be a geometrically irreducible algebraic
surface of Kodaira dimension $0$
over a field~$\Bbbk$ of characteristic $p$.
We may assume that~$\Bbbk$ is algebraically closed.
By Corollary~\ref{corollary:smooth-projective-model}
we can also assume that $S$ is smooth and projective.
Furthermore, by~\mbox{\cite[Proposition~1.6]{ProkhorovShramov-2019}}
it is enough to consider the case when $p>0$.

We can replace $S$ by its
minimal model, so that~\mbox{$\Bir(S)=\Aut(S)$}
by Lemma~\ref{lemma:Bir-vs-Aut}.
By Theorem~\ref{theorem:KE}(ii),
we need to provide bounds for the indices of normal abelian subgroups of $\Aut(S)$ in
the cases when $S$ is a~$K3$
surface, an Enriques surface, an abelian surface, a hyperelliptic
surface, or a quasi-hyperelliptic surface. In the first two cases this is done by Lemma~\ref{lemma:K3}.
In the third case this is done by Corollary~\ref{corollary:abelian-variety-Jordan}.
In the last two cases this is done
by Lemma~\ref{lemma:hyperelliptic-groups}(i).
\end{proof}

\begin{proof}[Proof of Proposition~\ref{proposition:stabilizer-kappa-0}]
Let $S$ be a smooth irreducible projective
surface of Kodaira dimension $0$ over a field~$\Bbbk$
of characteristic $p$, and let~$P$ be a $\Bbbk$-point on~$S$.
Replacing the surface $S$ by the surface $S_{\bar{\Bbbk}}$, and the $\Bbbk$-point $P$
by the $\bar{\Bbbk}$-point~$P_{\bar{\Bbbk}}$ of~$S_{\bar{\Bbbk}}$, we may assume that the field $\Bbbk$
is algebraically closed.
By~\mbox{\cite[Proposition~1.3]{ProkhorovShramov-2019}}
it is enough to deal with the case when~\mbox{$p>0$}.

Consider the minimal model  $S'$ of the surface $S$.
Then~\mbox{$\Aut(S)\subset\Bir(S')$}; hence by
Lemma~\ref{lemma:Bir-vs-Aut}
there is an embedding~\mbox{$\Aut(S)\subset\Aut(S')$}.
Consider the birational morphism~\mbox{$\pi\colon S\to S'$}.
Since the minimal model $S'$ is unique by
Lemma~\ref{lemma:Bir-vs-Aut}, the morphism~$\pi$
is equivariant with respect to the group~\mbox{$\Aut(S)$}.
Therefore, the image~\mbox{$\pi(P)$} of the point $P$ is invariant
under the group~\mbox{$\Aut(S;P)$}.
Thus we can assume from the very beginning that the surface $S$ is minimal.
By Theorem~\ref{theorem:KE}(ii),
we need to provide bounds for the order of~\mbox{$\Aut(S;P)$} in
the cases when $S$ is a~$K3$ surface, an Enriques surface,
an abelian surface, a hyperelliptic
surface, or a quasi-hyperelliptic surface. In the first two cases this is done by Lemma~\ref{lemma:K3}.
In the third case this is done by Corollary~\ref{corollary:torus-stabilizer}.
In the last two cases this is done
by Lemma~\ref{lemma:hyperelliptic-groups}(ii).
\end{proof}

There is a partial analog of Proposition~\ref{proposition:stabilizer-kappa-0} that
is valid over arbitrary fields. It was suggested to us by Yi Gu.

\begin{definition}
Let $X$ be an algebraic variety over a field $\Bbbk$, and let $P$ be a closed point of $X$.
The \emph{inertia group} $\Ine(X;P)$ of $P$ is the kernel of the action of the stabilizer $\Aut(X;P)$
on the residue field $\Bbbk(P)$.
\end{definition}

\begin{lemma}\label{lemma:inertia}
Let $X$ be an algebraic variety over a field $\Bbbk$, and let $P$ be a closed point of $X$.
Let $P_1$ be one of the $\bar{\Bbbk}$-points of $P_{\bar{\Bbbk}}$.
Then $\Ine(X;P)\subset\Aut(X_{\bar{\Bbbk}};P_1)$.
\end{lemma}

\begin{proof}
The action of $\Ine(X;P)\subset\Aut(X_{\bar{\Bbbk}})$ on
$$
P_{\bar{\Bbbk}}\cong \Spec\left(\Bbbk(P)\otimes_{\Bbbk} \bar{\Bbbk}\right)
$$
is trivial.
\end{proof}

\begin{corollary}\label{corollary:inertia}
There exists a constant $B$ such that for every
field $\Bbbk$, every smooth geometrically
irreducible projective surface
$S$ of Kodaira dimension~$0$ over $\Bbbk$, every
closed point $P\in S$, and every finite subgroup~\mbox{$G\subset \Ine(S;P)$}
the order of the group $G$ is at most~$B$.
\end{corollary}

\begin{proof}
Apply Proposition~\ref{proposition:stabilizer-kappa-0} together with Lemma~\ref{lemma:inertia}.
\end{proof}

\section{Fixed points in arbitrary dimension}
\label{section:high-dim}

In this section we prove Theorem~\ref{theorem:high-dimension-stabilizer}.
We start by recalling the rigidity theorem
for projective varieties, see \cite[\S\,III.4.3]{Shafarevich} or \cite[Lemma~3.3.3]{Brion-long}.

\begin{theorem}
\label{theorem:rigidity}
Let $U$, $V$, and $W$ be varieties over an algebraically closed field~$\Bbbk$.
Suppose that $U$ is irreducible, and $V$ is irreducible and projective.
Let
$$
f\colon U\times V\to W
$$
be a morphism. Suppose that for some $\Bbbk$-point $u_0\in U$ the subvariety
$$
\{u_0\}\times V\subset U\times V
$$
is mapped to a point
by $f$. Then for every $\Bbbk$-point $u\in U$ the subvariety
$\{u\}\times V$ is also mapped to a point.
\end{theorem}

\begin{corollary}\label{corollary:abelian-variety-no-fixed-points}
Let $X$ be a geometrically irreducible projective variety over a
field~$\Bbbk$, and let $\mathcal{A}\subset\AAut_X$ be a positive-dimensional abelian variety.
Then~$\mathcal{A}$ has no fixed (closed) points on~$X$.
\end{corollary}

\begin{proof}
Suppose that~$\mathcal{A}$ acts on $X$ with a fixed closed point $P$.
Since~$\mathcal{A}$ is (geometrically) connected, and $P_{\bar{\Bbbk}}$ is a finite union of $\bar{\Bbbk}$-points,
one can see that~$\mathcal{A}_{\bar{\Bbbk}}$
acts on $X_{\bar{\Bbbk}}$ with a fixed point as well. Hence, we may assume that the field~$\Bbbk$
is algebraically closed.

The action of $\mathcal{A}$ on $X$ is given by a morphism
$$
\Psi\colon \mathcal{A}\times X\to X.
$$
The image~\mbox{$\Psi(\mathcal{A}\times\{P\})$} is a point.
On the other hand, since the action of
$\mathcal{A}$ on $X$ is non-trivial, for a general $\Bbbk$-point $Q\in X$
the image $\Psi(\mathcal{A}\times\{Q\})$ is not a point.
This is impossible by Theorem~\ref{theorem:rigidity} because~$\mathcal{A}$ is projective.
\end{proof}

We will need the following simple but convenient fact.

\begin{lemma}\label{lemma:number-of-points}
Let $Z$ and $X$ be projective schemes over a
field $\Bbbk$, and let~\mbox{$\pi\colon Z\to X$} be a morphism such that every fiber
of $\pi$ is finite. Then there is a constant~$C=C(Z,X,\pi)$ such that
the number of $\Bbbk$-points in every fiber of~$\pi$ is at most~$C$.
\end{lemma}
\begin{proof}
The number of irreducible components of $Z$ is finite,
so to bound the number of $\Bbbk$-points in the fibers
we may replace~$Z$ by its irreducible component.
Furthermore, we can assume that $Z$ and $X$ are reduced (so that they are projective varieties).
Since~$Z$ is projective,
the morphism $\pi$ is projective, see e.g.~\cite[Corollary~3.3.32(e)]{Liu}.
Thus we conclude that~$\pi$
is a finite morphism, see for instance \cite[Corollary~4.4.7]{Liu}.
By generic flatness (see for instance~\mbox{\cite[\href{https://stacks.math.columbia.edu/tag/052B}{Tag 052B}]{Stack}}),
there exists a dense open subset $X^0\subset X$
such that the restriction $\pi^0$ of $\pi$ to $Z^0=\pi^{-1}(X^0)$ is flat. Thus the number of $\Bbbk$-points
in the fibers of $\pi^0$ is bounded by the degree of~$\pi^0$, see \cite[Corollary~III.9.10]{Hartshorne}.
Replacing $Z$ and $X$ by $Z\setminus Z^0$ and $X\setminus X^0$ and proceeding by Noetherian induction,
we obtain the assertion of the lemma.
\end{proof}

Now we prove Theorem~\ref{theorem:high-dimension-stabilizer}.

\begin{proof}[Proof of Theorem~\ref{theorem:high-dimension-stabilizer}]
The proof is identical to that of \cite[Theorem~1.5]{ProkhorovShramov-2019}.
The reduced neutral component~\mbox{$\AAut^0_{X,\red}$}
of the group scheme~\mbox{$\AAut_X$} is an abelian variety by Corollary~\ref{corollary:Aut-0-abelian}.
Let $\AAut^0_{X;P}$ be the stabilizer of $P$ in $\AAut^0_{X,\red}$. Then
$\AAut^0_{X;P}$ is a group scheme of finite type over $\Bbbk$ (note however that it may be non-reduced and not connected).
Denote by $\Aut^0(X;P)$ the group of $\Bbbk$-points of $\AAut^0_{X;P}$, so that $\Aut^0(X;P)$ is the stabilizer
of $P$ in $\Aut^0(X)$.

We claim that the group scheme $\AAut^0_{X;P}$ is finite. Indeed,
suppose that~\mbox{$\AAut^0_{X;P}$} is infinite.
Then it contains some
positive-dimensional abelian variety~$\mathcal{A}$.
Thus~$\mathcal{A}$ acts on $X$ with the fixed point $P$.
This is impossible by Corollary~\ref{corollary:abelian-variety-no-fixed-points}.

Now we know that the group scheme $\AAut^0_{X;P}$ is finite.
We claim that the order of the group $\Aut^0(X;P)$
is bounded by some constant $C(X)$ that does not depend on $P$.
Consider the incidence relation
$$
Z=\{(\sigma,Q)\mid \sigma(Q)=Q\}\subset \AAut^0_{X,\red}\times X,
$$
and denote by $\pi\colon Z\to X$ the projection to the second factor.
Then~$Z$ is a projective scheme,
and a fiber of~$\pi$ over a point $Q$
is exactly the group scheme~\mbox{$\AAut^0_{X;Q}$}.
Thus, the fibers of $\pi$ are finite.
Therefore, according to Lemma~\ref{lemma:number-of-points},
the number of $\Bbbk$-points in
the fibers of $\pi$, i.e. the order of the groups
$\Aut^0(X;Q)$, is bounded by a constant~\mbox{$C(X)$}.

Finally, we recall from Lemma~\ref{lemma:MZ}
that the quotient $\Aut(X)/\Aut^0(X)$ has bounded finite subgroups.
Hence the orders of the finite subgroups of the group
$$
\Aut(X;P)/\Aut^0(X;P)\subset\Aut(X)/\Aut^0(X)
$$
are bounded by some constant $B(X)$ that does not depend on $P$.
This implies the required assertion
about the group~\mbox{$\Aut(X;P)$}.
\end{proof}

\section{Nilpotent groups}
\label{section:nilpotent}

In this section we study finite nilpotent subgroups of
birational automorphism groups and prove Theorem~\ref{theorem:nilpotent}. Let us start with the case of
ruled surfaces over elliptic curves.

\begin{lemma}\label{lemma:nilpotent-ruled}
Let $\Bbbk$ be a field of characteristic $p>0$, and let $S$ be a geometrically irreducible algebraic surface over $\Bbbk$ birational to
$E\times\PP^1$, where $E$ is an elliptic curve.
Let~\mbox{$G\subset\Bir(S)$} be a finite subgroup, and let $G_p$ denote a $p$-Sylow subgroup of~$G$. Then~$G$ contains
either a normal abelian subgroup of order coprime to $p$ and index at most~\mbox{$2^{15}\cdot 3^5\cdot 5^4\cdot |G_p|^{15}$},
or a normal nilpotent subgroup of class at most $2$, order coprime to $p$, and index at most~\mbox{$2^9\cdot 3^2\cdot 5\cdot|G_p|^3$}.
\end{lemma}

\begin{proof}
We may assume that $\Bbbk$ is algebraically closed.
According to Lemma~\ref{lemma:ruled-Bir}, the group
$\Bir(S)$ fits into the exact sequence
\begin{equation*}
1\to \Bir(S)_\phi\to \Bir(S)\to\Gamma,
\end{equation*}
where $\Bir(S)_\phi\subset\PGL_2(\Bbbk(E))$
and $\Gamma\subset\Aut(E)$.
Thus, we obtain an exact sequence of finite groups
\begin{equation*}
1\to F\to G\to \bar{G}\to 1,
\end{equation*}
where $F$ is a subgroup of $\PGL_2(\Bbbk(E))$, and $\bar{G}$ acts faithfully on~$E$.
By Corollary~\ref{corollary:abelian-variety-p-Jordan} and Remark~\ref{remark:elliptic-curves-uniform} there exists a normal
subgroup $\bar{H}$ in~$\bar{G}$
such that~$\bar{H}$ is generated by at most $2$ elements, the order of $\bar{H}$ is coprime to~$p$, and the index of $\bar{H}$ in $\bar{G}$ does not exceed
$24|\bar{G}_p|$, where $\bar{G}_p$ is a $p$-Sylow subgroup of~$\bar{G}$.
Let~$H$ be the preimage of $\bar{H}$ in $G$, so that there
is an exact sequence of groups
$$
1\to F\to H\to \bar{H}\to 1.
$$
In particular, $H$ is a normal subgroup of $G$.
By Lemma~\ref{lemma:P1} there exists a characteristic cyclic
subgroup $R$ of order $n$ coprime to $p$
and index at most~\mbox{$60|F_p|^3$} in $F$, where $F_p$
is a $p$-Sylow subgroup of~$F$.

Suppose that $n\le 2$. Then $|F|\le 120|F_p|^3$.
If $F$ is a group of one of types~\mbox{(1)--(4)} in the notation of Theorem~\ref{theorem:ADE},
then it is generated by at most~$2$ elements by Remark~\ref{remark:PGL-generators}.
Hence $H$ contains a characteristic abelian subgroup~$A$ of order coprime to $p$ and index at most $120^4\cdot |H_p|^{13}$ by Lemma~\ref{lemma:Heisenberg-estimate-1}.
Thus $A$ is normal in $G$, and its index in~$G$ is at most
$$
24\cdot 120^4\cdot |\bar{G}_p|\cdot |H_p|^{13}\le  2^{15}\cdot 3^5\cdot 5^4\cdot |G_p|^{13}\le 2^{15}\cdot 3^5\cdot 5^4\cdot |G_p|^{15}.
$$
If $F$ is a group of type~(5) in the notation of Theorem~\ref{theorem:ADE},
then $H$ contains a characteristic abelian subgroup
$A$ of order coprime to $p$ and index at most~\mbox{$120^4\cdot |H_p|^{15}$} by Lemma~\ref{lemma:Heisenberg-estimate-2}.
Thus $A$ is normal in $G$, and its index in $G$ is at most
$$
24\cdot 120^4\cdot |\bar{G}_p|\cdot |H_p|^{15}\le  2^{15}\cdot 3^5\cdot 5^4\cdot |G_p|^{15}.
$$

Therefore, we can assume that $n\ge 3$.
The group $H$ acts on $F$ by conjugation,
and this action preserves the characteristic subgroup~$R$.
Pick two elements~$\alpha_1$ and~$\alpha_2$ of $H$
such that their images $\bar{\alpha}_1$ and $\bar{\alpha}_2$ in $\bar{H}$ generate $\bar{H}$.
Since each of the elements~$\alpha_i$ normalizes the subgroup $R$, and the (algebraically closed)
field~$\Bbbk$ contains a primitive root of $1$ or degree $n=|R|$, by Lemma~\ref{lemma:Serre}
the elements~$\alpha_i^2$ commute with~$R$.

Let $A$ be the subgroup of $H$ generated by $\alpha_1^2$ and $\alpha_2^2$. Then
$A_F=A\cap F$ is contained in the centralizer of $R$. Since $n\ge 3$, the group $A_F$
is cyclic and has order coprime to $p$ by Remark~\ref{remark:PGL2-centralizer}.
In particular, this means that the order of $A$ is coprime to $p$.
Furthermore, since $A_F$ is normalized by $\alpha_i^2$, we can use Lemma~\ref{lemma:Serre} once again
and conclude that $A_F$ commutes with $\alpha_1^4$ and $\alpha_2^4$.

Denote by $A'$  the subgroup of $H$ generated by $\alpha_1^4$ and $\alpha_2^4$,
and set~\mbox{$A'_F=A'\cap F$}. Then $A'_F\subset A_F$. Hence $A'_F$ is a cyclic central subgroup
of $A'$. Thus $A'$ is a central extension of an abelian group, which implies that
it is a nilpotent group of class at most $2$.
Let $\tilde{N}$ be the subgroup of $H$ generated by $A'$ and $R$. Since~$R$ is abelian and
commutes with $A'$, we see that $\tilde{N}$ is a nilpotent group of class at most $2$.
Moreover, its order is coprime to $p$. We want to replace $\tilde{N}$ by a subgroup with similar properties
that is normal in~$G$.

The subgroup $R$ is characteristic in $F$, and hence
normal in $G$.
Also, the subgroup $\bar{H}_{\bar{\alpha}_1^4, \bar{\alpha}_2^4}$
generated by $\bar{\alpha}_1^4$ and $\bar{\alpha}_1^4$ is characteristic in $\bar{H}$, and its index in $\bar{H}$ is at most~$16$,
see Example~\ref{example:characteristic-by-generators}.
Thus $\bar{H}_{\bar{\alpha}_1^4, \bar{\alpha}_2^4}$ is normal in $\bar{G}$.
Let $N$ be the intersection of all subgroups
of $G$ conjugate to $\tilde{N}$. Then $N$ is a nilpotent group of class at most $2$;
moreover, its order is coprime to $p$, and it is normal in~$G$.
Also, we know that $N$ contains the subgroup $R$, because~$R$ is normal in~$G$.
Similarly, the image of $N$ in $\bar{H}$ coincides with $\bar{H}_{\bar{\alpha}_1^4, \bar{\alpha}_2^4}$,
because~$\bar{H}_{\bar{\alpha}_1^4, \bar{\alpha}_2^4}$ is normal in $\bar{G}$.
Therefore, the index of
$N$ in $G$ is
\begin{multline*}
\frac{|G|}{|N|}=
\frac{|\bar{G}|}{|\bar{H}_{\bar{\alpha}_1^4, \bar{\alpha}_2^4}|}\cdot\frac{|F|}{|N\cap F|}\le
16\frac{|\bar{G}|}{|\bar{H}|}\cdot \frac{|F|}{|R|} \le
16\cdot \left(24\cdot |\bar{G}_p|^3\right)\cdot \left(60\cdot |F_p|^3\right)\\ =
2^9\cdot 3^2\cdot 5\cdot |F_p|^3\cdot |\bar{G}_p|^3=2^9\cdot 3^2\cdot 5\cdot |G_p|^3.
\end{multline*}
\end{proof}

Now we can prove Theorem~\ref{theorem:nilpotent}.

\begin{proof}[Proof of Theorem~\ref{theorem:nilpotent}]
We may assume that the field $\Bbbk$ is algebraically closed.
If~$S$ is not birational to the product $E\times\PP^1$, where $E$ is
an elliptic curve, then the group $\Bir(S)$ is $p$-Jordan by Theorem~\ref{theorem:Popov};
in particular, this means that~\mbox{$\Bir(S)$} is nilpotently
$p$-Jordan of class at most~$2$.
On the other hand, if $S$ is birational to such a product, then
$\Bir(S)$ is nilpotently
$p$-Jordan of class at most~$2$ by Lemma~\ref{lemma:nilpotent-ruled}.
\end{proof}

\section{Conclusion}
\label{section:discussion}

In this section we discuss some open questions concerning
birational automorphism groups of varieties over fields of positive
characteristic.

\medskip
\textbf{Cremona groups of higher rank.}
In \cite[6.1]{Serre-2009} J.-P.\,Serre asked whether
the groups $\Bir(\PP^n)$ over a field of characteristic $p$ are generalized
$p$-Jordan for arbitrary $n$.
Fei Hu strengthened this question in \cite[Question~1.11]{Hu} by asking whether
these groups are actually $p$-Jordan. The answer to both
questions is not known, but Theorem~\ref{theorem:Serre} gives a hope that
it should be positive.
Over fields of characteristic zero, we know
from \cite{ProkhorovShramov-Cr} that all the groups
$\Bir(\PP^n)$ (and more generally, all birational automorphism groups
of rationally connected varieties) are Jordan.
However, the proof of this fact given in \cite{ProkhorovShramov-Cr}
relies on several theorems that are not known to hold in positive
characteristic. Most importantly (except for
resoultion of singularities, which people are used to routinely assume),
it uses boundedness of terminal Fano
varieties, which is a particular case of the result of C.\,Birkar
\cite{Birkar}. Except for this, the proof uses the Minimal Model Program,
which is available only up to dimension $3$
over fields of characteristic $p>5$ (see \cite{Hacon-Xu-15},
\cite{Birkar-16}, \cite{Birkar-Waldron}),
and some properties of minimal centers of log canonical singularities.
It may be interesting to try to generalize this proof to the case
of positive characteristic \emph{assuming}
resolution of singularities, boundedness of Fanos, and the Minimal Model
Program. Even in this setup we expect it to be a complicated problem.

\medskip
\textbf{Non-uniruled varieties.}
The case of non-uniruled varieties looks much more accessible.
Based on \cite[Theorem~1.8(ii)]{ProkhorovShramov-Bir}, we ask
the following.

\begin{question}\label{question:non-uniruled}
Let $X$ be a non-uniruled geometrically irreducible algebraic variety over a field of
characteristic~\mbox{$p>0$}. Is it true that the group $\Bir(X)$
is $p$-Jordan?
\end{question}

Note that the approach to birational automorphism groups
of non-uniruled varieties over fields of characteristic zero used
in \cite{ProkhorovShramov-Bir}
does not require boundedness of Fanos, and also does not use
much of the Minimal Model Program (that is, does not require termination
of flips). Therefore, we expect that
the answer to Question~\ref{question:non-uniruled}
may be obtained using the method of~\cite{ProkhorovShramov-Bir}.

\medskip
\textbf{Nilpotent groups.}
Based on \cite{Guld} and Theorem~\ref{theorem:nilpotent},
one can ask

\begin{question}\label{question:nilpotent}
Let $X$ be a geometrically irreducible algebraic variety over a field of
characteristic $p>0$. Is it true that the group $\Bir(X)$
is nilpotently $p$-Jordan?
\end{question}

The approach to nilpotent Jordan property used in \cite{Guld}
is based on many features specific to characteristic zero.
In particular, it uses the results of \cite{ProkhorovShramov-Cr},
which in turn require boundedness
of Fanos etc.

\medskip
\textbf{Complete varieties.}
One can wonder if the analog of Theorem~\ref{theorem:Hu}
holds for automorphism groups of complete algebraic varieties.

\begin{question}\label{question:complete}
Let $X$ be a complete algebraic variety over a field
of characteristic~\mbox{$p>0$}.
Is it true that the group $\Aut(X)$ is $p$-Jordan?
\end{question}

The difficulty
here is that an analog of Lemma~\ref{lemma:MZ} is not known in this case,
even in characteristic~$0$. However, according to \cite[Corollary~1.2]{MPZ} the answer to an analog of Question~\ref{question:complete} is positive
over fields of characteristic~$0$
(cf. \cite{ProkhorovShramov-Moishezon} for an alternative proof in the three-dimensional case).

\medskip
\textbf{Quasi-projective varieties.}
Similarly to \cite[Question~2.30]{Popov},
one can ask the following.

\begin{question}\label{question:q-p}
Let $X$ be a quasi-projective variety over a field
of positive characteristic~$p$. Is it true that the group
$\Aut(X)$ is $p$-Jordan?
\end{question}

Note that the answer to
a similar question is known to be positive
for quasi-projective surfaces over fields of zero characteristic,
see~\cite{BZ}.

\medskip
\textbf{Explicit estimates.}
It would be interesting to find the precise value of the constant~$J$
appearing in Theorem~\ref{theorem:Serre}, similarly to what
was done in~\mbox{\cite[Proposition~1.2.3]{ProkhorovShramov-constant}}
and~\cite{Yasinsky}. As in the case of zero characteristic,
this will amount to accurate analysis of finite groups acting on del Pezzo
surfaces. Note that some automorphism groups of del Pezzo surfaces
over fields of small positive characteristic do not appear
in large characteristics and characteristic zero, see
for instance~\cite[Lemma~5.1]{DD}.
Thus it may happen that the resulting constants are different
for small and large values of the characteristic.

\medskip
\textbf{Exponents.}
It would be interesting to
find the exact values of the constants~$e(\Gamma)$ of Definition~\ref{D:g-Jordan-2}
for various groups~$\Gamma$ that enjoy the $p$-Jordan property.
This applies to algebraic groups and automorphism groups of projective varieties; see
\cite[Theorems~1.7 and~1.10]{Hu}, and cf.~\cite[Remark~1.8]{Hu} for a (possibly non-sharp) upper bound for~$e(\Gamma)$
in these cases.  Also, this applies to our Theorem~\ref{theorem:Popov}(ii).
Similarly, we do not know the values of the constants~$e(\Gamma)$ of Definition~\ref{D:g-nilp-Jordan-2}
for birational automorphism groups of surfaces, see Theorem~\ref{theorem:nilpotent}.
Provided that one understands an analogous bound for Theorem~\ref{theorem:Popov}, computing or bounding these values will
boil down to analyzing the proof of Lemma~\ref{lemma:nilpotent-ruled}.
Some of the progress in this direction may be made by optimizing the bounds provided
by Lemmas~\ref{lemma:Heisenberg-estimate-1} and~\ref{lemma:Heisenberg-estimate-2}.

%\medskip
%\textbf{Ranks of abelian subgroups.}
%It was shown in \cite[Theorem~5.3]{Serre-2009} that every finite
%subgroup of the Cremona group of rank $2$ whose order is coprime
%to the characteristic of the base field contains a normal abelian
%subgroup of bounded index, and  this abelian subgroup
%can be generated by at most two elements.
%It would be interesting to find out if a similar estimate for the
%number of generators holds in the context of
%Theorem~\ref{theorem:Serre} (cf. Lemma~\ref{lemma:conic-bundle-Jordan}).
%Note that in the case of
%automorphism groups of projective varieties
%over fields of characteristic zero (and also in the case of
%Jordan diffeomorphism groups of smooth compact manifolds)
%such a bound always exists, see~\cite[Theorem~1.3]{MiR}.

\medskip
\textbf{Generalized $p$-Jordan property.}
We are not aware of examples of algebraic
varieties over a field of positive characteristic $p$ whose birational automorphism
group is not $p$-Jordan, but is
generalized $p$-Jordan.
Theorem~\ref{theorem:Popov} shows that there are no such varieties in dimension~$2$.
It would be interesting to find out if examples of this kind exist in higher dimensions.

\medskip
\textbf{Multiplicative bounds.}
In certain cases it is useful to consider multiplicative bounds for indices
of normal abelian subgroups in finite subgroups of a given group (i.e., the least common multiples of such indices);
see \cite{Serre-2007}, \cite{Serre-2009},
and~\cite[Theorem~1.2(ii)]{SV21}.
We point out that in the context of $p$-Jordan groups such bounds fail
to exist already in the most simple situations, even if we consider such numbers up to the powers of~$p$.
For instance,
if~$\Bbbk$ is an algebraically closed field of positive characteristic~$p$,
the $p$-Jordan group~\mbox{$\PGL_2(\Bbbk)$} contains a subgroup $\PSL_2(\mathbf{F}_{p^k})$
for every positive integer~$k$. The latter group is simple if $p^k>3$, so that
the largest normal abelian subgroup therein is the trivial group, whose index equals
$|\PSL_2(\mathbf{F}_{p^k})|=n_kp^k$. Here~\mbox{$n_k=\frac{p^{2k}-1}{2}$} if~\mbox{$p\ge 3$}, and~\mbox{$n_k=2^{2k}-1$}
if $p=2$; thus, $n_k$ is coprime to~$p$. We see that the numbers~\mbox{$n_k$, $k\ge 1$},
are unbounded, and so they do not have a finite common multiple.
That being said, it would be interesting to find examples of $p$-Jordan groups
of geometric origin where the multiplicative bounds for the arising constants do exist.

%\medskip
%\textbf{Non-smooth and non-regular varieties.}
%Most of the results we prove in this paper
%concern the groups of birational automorphisms. Thus,
%for their proofs it is enough to assume that the base field is
%algebraically closed, and replace a variety by its smooth birational
%model (provided that resolution of singularities is available).
%This may be not the case for the results concerning
%automorphism groups, i.e. Propositions~\ref{proposition:Hu-improved}
%and~\ref{proposition:stabilizer-kappa-0}, and Theorem~\ref{theorem:high-dimension-stabilizer}.
%It would be interesting to find out whether one can get rid of the smoothness assumption
%in these statements. For Proposition~\ref{proposition:Hu-improved} and Theorem~\ref{theorem:high-dimension-stabilizer},
%we should assume that the Kodaira dimension of~$X$ is well defined; for instance, this would be the case if resolution of singularities
%is available over the algebraic closure of~$\Bbbk$.

\end{document}